\documentclass[12pt]{amsart} 
\usepackage{amsfonts,graphics,amsmath,amsthm,amsfonts,amscd, amssymb,amsmath,latexsym}
\usepackage{epsfig}
\usepackage{flafter}

\usepackage{mathtools}

\theoremstyle{plain}
\newtheorem{theorem}{Theorem}[section]
\newtheorem{corollary}[theorem]{Corollary}
\newtheorem{lemma}[theorem]{Lemma}

\newtheorem{proposition}[theorem]{Proposition}
\newtheorem{definition-lemma}[theorem]{Definition-Lemma}

\newtheorem{defn}[theorem]{Definition}
\newtheorem{remark}[theorem]{Remark}
\newtheorem{conjecture}[theorem]{Conjecture}

\theoremstyle{definition}

\newtheorem*{ack*}{Acknowledgements}

\title{On a generalized canonical bundle formula and generalized adjunction}
\author{Stefano Filipazzi}

\usepackage[english,italian]{babel}
\usepackage{graphicx,subfigure}
\usepackage[a4paper,top=1in,bottom=1in,left=1in,right=1in]{geometry}
\usepackage{tikz}
\usetikzlibrary{matrix,arrows,decorations.pathmorphing}

\usepackage{hyperref}

\usepackage[babel]{csquotes}
\usepackage[maxnames=99,style=numeric,backend=bibtex]{biblatex}
\newcommand{\Q}{\mathbb{Q}}			
\newcommand{\R}{\mathbb{R}}			
\newcommand{\C}{\mathbb{C}}			
\newcommand{\N}{\mathbb{N}}			
\newcommand{\rar}{\rightarrow}		
\newcommand{\drar}{\dashrightarrow}	

\DeclareMathOperator{\Spec}{Spec}		
\DeclareMathOperator{\mult}{mult}		
\DeclareMathOperator{\Supp}{Supp}		
\DeclareMathOperator{\codim}{codim}	
\DeclareMathOperator{\rk}{rk}			
\DeclareMathOperator{\coeff}{coeff}	
\DeclareMathOperator{\red}{red}		
\DeclareMathOperator{\Ex}{Ex}		
\DeclareMathOperator{\dv}{div}		
\DeclareMathOperator{\lct}{lct}		

\def\O#1.{\mathcal {O}_{#1}}			
\def\pr #1.{\mathbb P^{#1}}				
\def\af #1.{\mathbb A^{#1}}				
\def\ses#1.#2.#3.{0\to #1\to #2\to #3 \to 0}		
\def\xrar#1.{\xrightarrow{#1}}			
\def\K#1.{K_{#1}}						
\def\bA#1.{\mathbf{A}_{#1}}				
\def\bM#1.{\mathbf{M}_{#1}}				
\def\bL#1.{\mathbf{L}_{#1}}				
\def\bB#1.{\mathbf{B}_{#1}}				
\def\bK#1.{\mathbf{K}_{#1}}				
\def\bD#1.{\mathbf{D}_{#1}}				
\def\subs#1.{_{#1}}						
\def\sups#1.{^{#1}}						

\usepackage{lipsum}

\newcommand{\Addresses}{{
  \bigskip
  \footnotesize

  S.~Filipazzi, \textsc{Department of Mathematics, University of Utah,
    Salt Lake City,\\ UT 84112, USA}\par\nopagebreak
  \textit{E-mail address}: \texttt{filipazz@math.utah.edu}
  
}}

\begin{document}

\selectlanguage{english}

\begin{abstract}
In this note, we extend the theories of the canonical bundle formula and adjunction to the case of generalized pairs. As an application, we study a particular case of a conjecture by Prokhorov and Shokurov.
\end{abstract}

\maketitle

\tableofcontents

\section{Introduction}

Recently, Birkar and Zhang introduced the notion of generalized pair \cite{BZ}.
This kind of pair arises naturally in certain situations, such as the canonical bundle formula \cite{Kaw98, Amb99, FM00}, and adjunction theory \cite{Kaw98, Amb99, Bir16a}.
Furthermore, generalized pairs play an important role in recent developments, such as the study of the Iitaka fibration \cite{BZ}, and the proof of the BAB conjecture \cite{Bir16a, Bir16b}.

Among the techniques in birational geometry, adjunction theory is one of the most powerful tools.
It relates the geometry and the singularities of the ambient variety to those of appropriate subvarieties.
We call {\it adjunction} the process of inferring statements about a subvariety from some knowledge of the ambient variety, while the inverse and usually more complicated process is called {\it inversion of adjunction}.
The most satisfactory formulation of this theory in the case of pairs is the following, due to Hacon \cite{Hac14}.

\begin{theorem}[{\cite[Theorem 0.1]{Hac14}}] \label{inversion Hacon}
Let $W$ be a log canonical center of a pair $(X,\Delta = \sum \delta_i \Delta_i)$ where $0 \leq \delta_i \leq 1$.
Then $(X,\Delta)$ is log canonical in a neighborhood of $W$ if and only if $(W,\mathbf{B}(W;X,\Delta))$ is log canonical.
\end{theorem}

In the case that $W$ has codimension 1, the statement takes the following simpler form, originally due to Kawakita \cite{Kaw07}.

\begin{theorem}[\cite{Kaw07}] \label{inversion Kawakita}
Let $(X,S+B)$ be a log pair such that $S$ is a reduced divisor which has no common component with the support of $B$, let $S^\nu$ denote the normalization of $S$, and let $B^\nu$ denote the different of $B$ on $S^\nu$.
Then $(X,S+B)$ is log canonical near $S$ if and only if $(S^\nu,B^\nu)$ is log canonical.
\end{theorem}

In the setup of generalized pairs, Birkar has a version of divisorial inversion of adjunction under some technical conditions.

\begin{theorem}[Lemma 3.2, \cite{Bir16a}] \label{Birkar inversion of adjuction}
Let $(X',B'+M')$ be a $\Q$-factorial generalized pair with data $X \rar X'$ and $M$.
Assume $S'$ is a component of $B'$ with coefficient 1, and that $(X',S')$ is plt.
Let
$$
\K S'. + B_{S'} +M_{S'}= (\K X'. + B' +M')|_{S'}
$$
be given by generalized adjunction.
If $(S',B_{S'} +M_{S'})$ is generalized log canonical, then $(X',B'+M')$ is generalized log canonical near $S'$.
\end{theorem}

The purpose of this work is to improve the statement of Theorem \ref{Birkar inversion of adjuction} and broaden the current knowledge of inversion of adjunction in the setup of generalized pairs.
As the work of Birkar and Zhang does not consider adjunction for generalized log canonical centers of higher codimension \cite{BZ}, a relevant part of this note is to develop an appropriate theory in such setup.

In analogy to the work of Kawamata and Ambro \cite{Kaw98, Amb99}, we first define generalized adjunction in the case of fibrations.
Indeed, the canonical bundle formula is the key tool to define adjunction on higher codimensional centers.
In particular, we prove the following, which partly answers a question posed by Di Cerbo and Svaldi \cite[Remark 7.4]{dCS17}.

\begin{theorem} \label{full generality theorem firbations}
Let $(X',B'+M')$ be a {projective} generalized sub-pair with data $X \rar X'$ and $M$.
Assume that $B'$, $M'$ and $M$ are $\Q$-divisors.
Let $f  \colon  X' \rar Z'$ be a contraction such that $\K X'. + B' + M' \sim \subs \Q, f. 0$.
Also, let $(X',B'+M')$ be generalized log canonical over the generic point of $Z'$.
Then, the b-divisor $\bM Z'.$ is $\Q$-Cartier and b-nef.
\end{theorem}

The key step towards the proof of Theorem \ref{full generality theorem firbations} is the partial version given in Theorem \ref{lemma klt fibrations}.
The main tool in the proof of the latter is the weak semi-stable reduction introduced by Abramovich and Karu \cite{Kar99,AK}.

A suitable theory for a generalized canonical bundle formula allows us to move our focus to higher codimensional generalized log canonical centers.
First, we introduce an appropriate definition of adjunction in this setup.
The main idea is the following: let $W'$ be a generalized log canonical center of a generalized pair $(X',B'+M')$, and fix a generalized log canonical place $E \subset X$ on a higher birational model.
Thus, $E$ inherits a structure of generalized pair from divisorial generalized adjunction on $X$.
Then, we consider the fiber space $E \rar W'$ and induce a generalized pair structure on $W'$.
In particular, the following result can be seen as a generalization of Kawamata's subadjunction \cite{Kaw98}.

\begin{theorem} \label{gen adj higher codim}
Let $(X',B'+M')$ be a generalized pair with data $X \rar X' {\rar V}$ and $M$.
Let $B'$, $M'$ and $M$ be $\Q$-divisors.
Assume $(X',B'+M')$ is generalized log canonical, and fix an exceptional generalized log canonical center $W' \subset X'$.
{Assume that $W'$ is projective.}
Then, $W'$ is normal, and it admits a structure of generalized pair $(W',\bB W'. + \bM W'.)$.
In particular, the b-divisor $\bM W'.$ is a b-nef $\Q$-Cartier b-divisor, and $(W',\bB W'. + \bM W'.)$ is generalized klt.
\end{theorem}

Once generalized adjunction is established, we focus on generalized inversion of adjunction.
Following ideas of Hacon \cite{Hac14}, we prove the following.

\begin{theorem} \label{prop inv of adj}
Let $(X',B'+M')$ be a {projective} generalized pair with data $X \rar X'$ and $M$.
Assume that $B'$, $M$ and $M'$ are $\Q$-divisors.
Let $W'$ be a generalized log canonical center of $(X',B'+M')$ with normalization $W^\nu$.
{Assume that a structure of generalized pair $(W^\nu,\bB W^\nu. + \bM W^\nu.)$ is induced on the normalization $W^\nu$ of $W'$.}
Then, $(W^\nu,\bB W^\nu. + \bM W^\nu.)$ is generalized log canonical if and only if $(X',B'+M')$ is generalized log canonical in a neighborhood of $W'$.
\end{theorem}

Here, the main ingredients are the MMP and Kawamata--Viehweg vanishing.
In particular, Birkar and Zhang have developed an MMP in the setup of generalized pairs \cite{BZ}, and we apply such machinery to our particular case.
In general, the statements concerning adjunction theory are proved by considering a suitable higher model of the starting variety, where the divisors carrying discrepancy at most $-1$ have (close to) simple normal crossing configuration.
Once such a convenient arrangement is reached, the negativity lemma and Kawamata--Viehweg vanishing apply.
In the first formulations, the higher model is a log resolution \cite[cf. Theorem 5.50]{KM}, and subsequently, the notion of dlt model took place \cite{Hac14}.
In this note, we introduce an appropriate generalization of the latter.

Finally, we discuss some applications of the generalized canonical bundle formula to a famous conjecture by Prokhorov and Shokurov \cite[Conjecture 7.13]{PS}.
We prove some inductive statements, which allow reducing parts of the conjecture to some particular cases. This leads to some progress towards the conjecture for fibrations of relative dimension 2.

\begin{theorem} \label{prokhorov shokurov rel dim 2}
Let $(X,B)$ be a sub-pair, with $\coeff (B) \in \Q$.
Let $f \colon X \rar Z$ be a projective surjective morphism of normal varieties with connected fibers.
Assume $\K X. + B \sim \subs \Q,f. 0$, and $(X,B)$ is klt over the generic point of $Z$.
If the geometric generic fiber $X_{\overline{\eta}}$ is a surface not isomorphic to $\pr 2.$, then the b-divisor $\bM Y.$ is b-semi-ample.
\end{theorem}

The proof of Theorem \ref{prokhorov shokurov rel dim 2} relies on work of Shokurov and Prokhorov, who considered the case of relative dimension 1 \cite[Theorem 8.1]{PS}, and work of Fujino, who proved the statement when the fibers are surfaces of Kodaira dimension 0 \cite{Fuj03}.
Thus, excluding $\pr 2.$, we are left with considering fibrations whose geometric generic fiber, up to taking the minimal resolution, admits a morphism to a curve.
Under this condition, we are able to perform an inductive argument.

{
Under certain technical conditions, we can formulate Theorem \ref{prokhorov shokurov rel dim 2} to also address the case when the geometric generic fiber is $\pr 2.$.
In particular, if the generic fiber $(X_\eta,B_\eta)$ is not terminal, its terminalization is a pair $(X'_\eta,B'_\eta)$ such that $X'_{\overline{\eta}}$ admits a morphism to a curve, and $B'_\eta \geq 0$.
Thus, we can apply the strategy illustrated above.
For the reader's convenience, we include this alternative version of Theorem \ref{prokhorov shokurov rel dim 2} as a separate statement.}

\begin{theorem} \label{prokhorov shokurov rel dim 2 revised}
{Let $(X,B)$ be a sub-pair, with $\coeff (B) \in \Q$. Let $f \colon X \rar Z$ be a projective surjective morphism of normal varieties with connected fibers and $\dim X - \dim Z =2$.
Assume $\K X. + B \sim \subs \Q,f. 0$, and $(X,B)$ is klt but not terminal over the generic point of $Z$.
Then, the b-divisor $\bM Y.$ is b-semi-ample.}
\end{theorem}

After reviewing some facts about generalized pairs, we introduce the notion of {\it weak generalized dlt model}, which carries analogs to most of the good properties of dlt models \cite[cf. Definitions and Notation 1.9]{KK10}.
In Theorem \ref{generalized weak dlt model} we prove that such models exist.
Then, we switch the focus to the generalized canonical bundle formula.
Once it is established, we apply this machinery to the study of generalized adjunction and inversion of adjunction.
We conclude discussing some applications to the conjecture by Prokhorov and Shokurov.

\begin{ack*}
The author would like to thank his advisor Christopher D. Hacon for suggesting the problem, for his insightful suggestions and encouragement.
He benefited from several discussions with Joaqu\'{i}n Moraga and Roberto Svaldi.
He would also like to thank Tommaso de Fernex and Karl Schwede for helpful conversations.
He is also grateful to Dan Abramovich, Florin Ambro and Kalle Karu for answering his questions.
He would like to thank Joaqu\'{i}n Moraga for useful remarks on a draft of this work, and Jingjun Han for pointing out the relation between the main result of this work and a theorem of Chen and Zhang.
Finally, he would like to express his gratitude to the anonymous referee for the careful report and the many suggestions.
The author was partially supported by NSF research grants no: DMS-1300750, DMS-1265285 and by a grant from the Simons Foundation; Award Number: 256202.
\end{ack*}

\section{Some notions about generalized pairs}

Throughout this paper, we will work over an algebraically closed field of characteristic 0.
In this section, we review some notions about generalized pairs. To start, we recall the definition of {\it pair} and {\it generalized pair}.

\begin{defn} \label{def gpair}{\em
A \emph{generalized (sub)-pair} is the datum of a normal variety $X'$, equipped with {projective} morphisms $X \rar  X' \rar V$, where $f \colon X \rar X'$ is birational and $X$ is normal, an $\R$-(sub)-boundary $B'$, and an $\R$-Cartier divisor $M$ on $X$ which is nef over $V$ and such that $\K X'.+B'+M'$ is $\R$-Cartier, where $M' \coloneqq f_*M$.
We call $B'$ the \emph{boundary part}, and $M'$ the \emph{nef part}.}
\end{defn}

\begin{remark}{\em
We have boundary and moduli b-divisors $\bB X'.$ and $\bM X'.$ naturally associated to a generalized pair.
Their traces on a higher model $\tilde X$ are denoted by $\bB X',\tilde{X}.$ and $\bM X' ,\tilde{X}.$ respectively.
In particular, the moduli part is the $\R$-Cartier b-divisor associated to $M$, denoted by $\overline{M}$.
We say it descends to $\tilde X$ whenever $\bM X'. = \overline{\bM X', \tilde{X}.}$.}
\end{remark}

\begin{remark}{\em
We recover the usual notion of (sub)-pair in the case $X=X'$, $M=M'=0$.
Also, the variety $V$ in Definition \ref{def gpair} is introduced to work in the relative setting.
As it does not contribute to the singularities of the generalized pair, unless otherwise stated, we will consider the absolute case $V = \Spec \C$, and we will omit it in the notation.
{Notice that, if $V= \Spec(\C)$, $X$ and $X'$ are projective varieties.}}
\end{remark}

Now, consider a generalized pair $(X',B'+M')$ with data $f \colon  X \rar X' {\rar V}$ and $M$.
Fix a divisor $E$ over $X'$.
As we are free to replace $X$ with a higher model, we may assume that $E$ is a divisor on $X$ itself.
Then, we can write
$$
\K X. + B + M = f^*(\K X'. + B' + M'),
$$
where $B$ is implicitly defined by the above equation and the choice $f_* \K X. = \K X'.$.
Then, the {\it generalized discrepancy} $a_E(X',B'+M')$ of $E$ with respect to $(X',B'+M')$ is $-b_E$, where $b_E$ is the coefficient of $E$ in $B$.
We say that $(X',B'+M')$ is {\it generalized log canonical}, in short {\it glc}, (respectively {\it generalized Kawamata log terminal}, in short {\it gklt}) if $a_E(X',B'+M') \geq -1$ (respectively $a_E(X',B'+M') > -1$) for any such $E$.

A subvariety $W' \subset X'$ is called \emph{generalized non-klt center} if there is a log resolution of $(X',B'+M')$ where $M$ descends, which we may assume to be $f \colon X \rar X'$ itself, such that $B = \sum b_i B_i$ and $\max \lbrace b_i | f(B_i) = W' \rbrace \geq 1$.
We say $W'$ is a \emph{generalized log canonical center} if $\max \lbrace b_i | f(B_i) = W' \rbrace = 1$.
In this situation, $(X',B'+M')$ is generalized log canonical in a neighborhood of the generic point of $W'$ \cite[cf. Proposition 17.1.1]{K92}.
Any divisor $E$ with $a_E(X',B'+M') \leq -1$ dominating a generalized non-klt (log canonical) center $W'$ is called \emph{generalized non-klt (log canonical) place}.
We say that $W'$ is an \emph{exceptional generalized log canonical center} if it is a generalized log canonical center admitting just one generalized log canonical place $E \subs W'.$ and such that the image of any other generalized non-klt place is disjoint from $W'$.

We say that $(X',B'+M')$ is \emph{generalized dlt} if $(X',B')$ is dlt, and every generalized non-klt center of $(X',B'+M')$ is a non-klt center of $(X',B')$.
If, in addition, every connected component of $\lfloor B' \rfloor$ is irreducible, we say $(X',B'+M')$ is \emph{generalized plt}.
Notice that a generalized pair might be generalized dlt but not generalized log canonical, as the moduli part may introduce deeper singularities over higher codimensional strata of $\lfloor B' \rfloor$.
On the other hand, if $(X',B'+M')$ is generalized plt, then it is generalized log canonical.

\begin{remark}{\em
In the case of usual pairs, i.e., $X=X'$, $M=M'=0$, the notion of generalized discrepancy recovers the classic notion of discrepancy and the corresponding measures of singularities.}
\end{remark}

{Let $(X',B'+M')$ be a generalized pair with data $X' \rar X \rar V$ and $M$.
Let $D'$ be an effective $\R$-divisor on $X'$ and $N$ an $\R$-Cartier divisor on $X$ that is nef over $V$.
Further, assume that $D'+N'$ is $\R$-Cartier, where $N'$ denotes the pushforward of $N$ to $X'$.
Then, the \emph{generalized log canonical threshold of $D'+N'$ with respect to $(X',B'+M')$} is defined as}
$$
{\rm glct}(\K X'.+B'+M'; D'+N') \coloneqq \sup\{t \mid \K X'.+B'+M'+t(D'+N')\text{ is glc} \},
$$
{where $(X',B'+M'+t(D'+N'))$ is considered as a generalized pair with boundary part $B'+tP'$
and moduli part $M'+tN'$. 
If the above set is empty, then we define the generalized log canonical threshold to be $-\infty$.
Observe that ${\rm glct}(\K X'.+B'+M'; D'+N')$ is non-negative provided that $\K X'.+B'+M'$ is generalized log canonical.
Moreover, ${\rm glct}(\K X'.+B'+M'; D'+N')$ is infinite if and only if $N$ descends on $X'$ and $D'$ is trivial.
}

We can now review the notion of generalized adjunction, first introduced in \cite{BZ}.
Fix a generalized pair $(X',B'+M')$ with data $f \colon  X \rar  X' {\rar V}$ and $M$.
Let $S'$ be an irreducible component of $B'$ of coefficient one, and denote by $S^\nu$ its normalization.
Up to replacing $X$ with a higher model, we may also assume that $X$ is a log resolution of $(X',B')$.
Denote by $g  \colon  S \rar S^\nu$ the induced morphism, where $S$ represents the strict transform of $S'$ on $X$.

As usual, we write
$$
\K X. + B + M = f^* (\K X'. + B' +M').
$$
Then, we set
$$
\K S. + B_S + M_S \coloneqq (\K X. + B +M)|_S,
$$
where $B_S \coloneqq (B-S)|_S$ and $M_S \coloneqq M|_S$. 
Define $B_{S^\nu} \coloneqq g_* B_S$, and $M_{S^\nu} \coloneqq g_* M_S$. By construction, we get
$$
\K S^\nu. + B_{S^\nu} + M_{S^\nu} = (\K X'. + B' + M')|_{S^\nu}.
$$
We refer to such operation as {\it generalized divisorial adjunction}.
As discussed in \cite[Definition 4.7]{BZ}, in the case $(X',B'+M')$ is generalized log canonical, the divisor $B \subs S^\nu.$ is effective on $S^\nu$, and therefore $(S^\nu,B \subs S^\nu. + M \subs S^\nu.)$ is a generalized pair with data $g \colon  S \rar S^\nu$ and $M_S$.

As mentioned in the introduction, generalized pairs arise naturally in the context of the canonical bundle formula.
Such construction was first introduced by Kawamata \cite{Kaw98}, and then widely studied by Ambro \cite{Amb99}, and Fujino and Mori \cite{FM00}.
We will just recall the main features of the so-called {\it adjunction for fiber spaces}, and refer to \cite{Bir16a} for a more complete exposition.

Let $(X,B)$ be a sub-pair, and let $f \colon  X \rar Z$ be a \emph{contraction} (i.e., a projective morphism such that $f_* \O X. = \O Z.$), where $\dim Z > 0$.
Assume that $(X,B)$ is sub-log canonical near the generic fiber of $f$, and that $\K X. + B \sim_{\R,f} 0$.
For each prime divisor $D$ on $Z$, let $t_D$ be the log canonical threshold of $f^*D$ with respect to $(X,B)$ over the generic point of $D$.
As $Z$ is normal, it is smooth along the generic point $\eta_D$ of $D$; therefore, $f^*D$ is well defined in a neighborhood of $\eta_D$, and the definition of $t_D$ is well posed.

Then, set $B_Z \coloneqq \sum b_D D$, where $b_D \coloneqq 1 - t_D$.
Notice that $b_D = 0$ for all but finitely many prime divisors on $Z$.
By assumption, we can find an $\R$-Cartier divisor $L_Z$ such that $\K X. + B \sim_\R f^* L_Z$.
Define $M_Z \coloneqq L_Z - (\K Z. + B_Z)$.
Thus, we have
$$
\K X. + B \sim_\R f^*(\K Z. + B_Z + M_Z).
$$
As $L_Z$ is defined just up to $\R$-linear equivalence, so is $M_Z$.
On the other hand, $B_Z$ is an honest $\R$-divisor on $Z$.

Taking higher models $X'$ and $Z'$ of $X$ and $Z$ respectively, one can induce divisors $B_{Z'}$ and $M_{Z'}$ on $Z'$.
These agree with $B_Z$ and $M_Z$ under pushforward, thus defining Weil b-divisors $\bB Z.$ and $\bM Z.$.
In particular, if $Z'$ is sufficiently high, $M_{Z'}$ is pseudoeffective.
Furthermore, under certain natural conditions, the b-divisor $\bM Z.$ is a b-nef $\Q$-Cartier b-divisor \cite[cf. Theorem 3.6]{Bir16a}.

\section{Weak generalized dlt models}

In this section, we introduce suitable modifications of a given generalized pair.
In order to do so, we need to recall the corresponding construction in the case of usual pairs.
We refer to \cite{KK10} for a more detailed discussion of the topic.
{The results of this sections hold for arbitrary generalized pairs $(X',B'+M')$ with data $X \rar X' \rar V$ and $M$, without the assumption $V= \Spec(\C)$.}

Let $(X,\Delta)$ be a pair, and let $f^m \colon X^m \rar X$ be a proper birational morphism whose exceptional locus $\Ex (f^m)$ is purely divisorial.
Let $\lbrace E_i \rbrace \subs i = 1. ^n$ denote the set of irreducible exceptional divisors, and let $\lbrace a_i \rbrace \subs i = 1. ^n$ denote the corresponding discrepancies.
Define $\Delta^m \coloneqq (f^m)_* \sups -1.(\Delta \wedge \Supp(\Delta)) + \sum \subs a_i \leq -1. E_i$, where the symbol $\wedge$ denotes the following operation.
Given two divisors $D_1= \sum_{i=1}^n d_i P_i$ and $D_2 = \sum_{i=1}^n e_i P_i$, we define $D_1 \wedge D_2 \coloneqq \sum_{i=1}^n \min \lbrace d_i,e_i \rbrace P_i$.
Then, $(X^m,\Delta^m)$ is a {\it minimal dlt model} of $(X,\Delta)$ if it is a dlt pair and the discrepancy of every $f^m$-exceptional divisor is at most $-1$. The existence of such models is due to Hacon.

\begin{theorem}[{\cite[Theorem 3.1]{KK10}}] \label{dlt model Hacon}
Let $(X,\Delta)$ be a pair such that $X$ is quasi-projective, $\Delta$ a boundary, and $\K X. + \Delta$ a $\Q$-Cartier divisor.
Then $(X,\Delta)$ admits a $\Q$-factorial minimal dlt model $f^m  \colon  (X^m,\Delta^m) \rar (X, \Delta)$.
\end{theorem}

In the setup of generalized pairs, we prove the following, which is a generalization of Theorem \ref{dlt model Hacon}.

\begin{theorem} \label{generalized weak dlt model}
Let $(X',B'+M')$ be a generalized pair with data $X \rar X' { \rar V}$ and $M$.
Then, there exists a $\Q$-factorial model $f^m \colon  X^m \rar X'$ such that every $f^m$-exceptional divisor has generalized discrepancy with respect to $(X',B'+M')$ at most $-1$.
Furthermore, the pair $(X^m,B^m)$ is dlt, where $B^m \coloneqq (f^m)^{-1}_* (B \wedge \Supp(B)) + E^m$, and $E^m$ denotes the reduced $f^m$-exceptional divisor.
\end{theorem}

\begin{remark}{\em
The construction in Theorem \ref{generalized weak dlt model} produces a generalized pair $(X^m,B^m +M^m)$ with data $X \rar X^m {\rar V}$ and $M$ \footnote{In general, we need to replace the $X$ and $M$ appearing in the statement of the theorem with higher models for $X \rar X^m$ to be a morphism. This is clear from the proof.}.
By construction, the singularities of $(X^m,B^m +M^m)$ are milder than the ones of $(X',B'+M')$.
Nevertheless, the construction does not guarantee that $(X^m,B^m +M^m)$ is generalized dlt.
Therefore, we call $(X^m,B^m +M^m)$ a {\it{weak generalized dlt model}} for $(X',B'+M')$.}
\end{remark}

\begin{proof}
Let $f \colon  X \rar X'$ be a log resolution of $(X,B)$ where $M$ descends.
Also, assume it is obtained by blowing up loci of codimension at least two.
In this way, there exists an effective and $f$-exceptional divisor $C$ such that $-C$ is $f$-ample.

Set $\lbrace B' \rbrace \coloneqq B' - \lfloor B' \rfloor$, and define $B$ via the identity
$$
\K X. + B + M = f^*(\K X'. + B' + M' ).
$$
Then, we can decompose $B$ as $B=f_*^{-1}\lbrace B' \rbrace + E^+ + F - G$, where $E^+$ denotes the (non necessarily $f$-exceptional) divisors with generalized discrepancy at most $-1$, $F$ the sum of all $f$-exceptional divisors with generalized discrepancy in $(-1,0]$, and $G$ the sum of all $f$-exceptional divisors with positive generalized discrepancy.
Also, define $E \coloneqq \red E^+$.

Let $H'$ be a sufficiently ample divisor on $X'$.
Then, for all $\epsilon, \mu, \nu \in \R$, we have
\begin{equation} \label{equation with epsilon mu nu}
E + (1 + \nu)F + \mu(-C+f^*H')+M = (1-\epsilon \mu)E + (1+ \nu)F + \mu(\epsilon E - C + f^*H') + M.
\end{equation}

If $0< \epsilon \ll 1$ and $\mu > 0$, both $\mu(-C+f^*H')+M$ and $\mu(\epsilon E - C + f^*H') + M$ are ample {over $X'$}, hence $\R$-linearly equivalent {over $X'$} to divisors $H_{1,\mu}$ and $H_{2,\mu}$ such that $B+H_{1,\mu}+H_{2,\mu}$ has simple normal crossing support, and $\lfloor H_{1,\mu} \rfloor = \lfloor H_{2,\mu} \rfloor =0$.

Thus, if $0 < \mu <1 $ and $0 < \nu \ll1$, the pair
$$
(X,f^{-1}_*\lbrace B' \rbrace+(1 - \epsilon \mu)E + (1 + \nu)F + H_{2,\mu})
$$
is klt.
By \cite{BCHM}, it has a $\Q$-factorial minimal model over $X'$
$$
f^m_{\epsilon,\mu,\nu}  \colon  (X^m_{\epsilon,\mu,\nu}, \Delta^m_{\epsilon,\mu,\nu}) \rar X'.
$$
In virtue of identity \eqref{equation with epsilon mu nu}, $f^m_{\epsilon,\mu,\nu}$ is also a minimal model for the pair
$$
(X,f^{-1}_*\lbrace B' \rbrace+ E + (1 + \nu)F + H_{1,\mu}).
$$
As the dlt property is preserved under steps of the MMP \cite[Corollary 3.44]{KM}, the resulting model is dlt as well.
Hence, the pair $(X^m_{\epsilon,\mu,\nu},B^m_{\epsilon,\mu,\nu})$ is dlt, where $B^m_{\epsilon,\mu,\nu}$ denotes the strict transform of $f^{-1}_* \lbrace B' \rbrace +E +F$ on $X^m_{\epsilon,\mu,\nu}$.

Now, let $\Gamma_{\epsilon,\mu,\nu}^m$ be the strict transform on $X_{\epsilon,\mu,\nu}^m$ of any other divisor $\Gamma$ on $X$.
Then, define
$$
N \coloneqq \K X_{\epsilon,\mu,\nu}^m. + B^m_{\epsilon,\mu,\nu} + \nu F_{\epsilon,\mu,\nu}^m + H^m_{1,\epsilon,\mu,\nu} \sim_\R \K X_{\epsilon,\mu,\nu}^m. + \Delta^m_{\epsilon,\mu,\nu},
$$
and
$$
T \coloneqq \K X_{\epsilon,\mu,\nu}^m. + B^m_{\epsilon,\mu,\nu} + (E^+-E)_{\epsilon,\mu,\nu}^m - G^m_{\epsilon,\mu,\nu} + M^m_{\epsilon,\mu,\nu} \sim_\R (f^m_{\epsilon,\mu,\nu})^* ( \K X'. + B' + M').
$$
The first one is $f^m_{\epsilon,\mu,\nu}$-nef, while the latter one is $f^m_{\epsilon,\mu,\nu}$-trivial.
Their difference can be written as
$$
T - N \sim_{\R,f^m_{\epsilon,\mu,\nu}} \mu C^m + (E^+-E)^m_{\epsilon,\mu,\nu} - G^m_{\epsilon,\mu,\nu} - \nu F^m_{\epsilon,\mu,\nu} \eqqcolon D^m_{\epsilon,\mu,\nu}.
$$
In particular, $-D^m_{\epsilon,\mu,\nu}$ is $f_{\epsilon,\mu,\nu}^m$-nef and $f_{\epsilon,\mu,\nu}^m$ exceptional.
Therefore, by the negativity lemma \cite[Lemma 3.39]{KM}, $D_{\epsilon,\mu,\nu}^m$ is effective.

As $C$, $E^+-E$, $F$ and $G$ are independent of ${\epsilon,\mu,\nu}$, if we choose $0 < \mu \ll \nu \ll 1$, both $G^m_{\epsilon,\mu,\nu}$ and $\nu F^m_{\epsilon,\mu,\nu}$ vanish \footnote{As $\mu \ll \nu$, the contribution of $\mu C^m$ is negligible in order to determine effective and anti-effective parts of $D^m \subs \epsilon,\mu,\nu.$. Since $D^m \subs \epsilon,\mu,\nu. \geq 0$, then $G^m_{\epsilon,\mu,\nu}$ and $ F^m_{\epsilon,\mu,\nu}$ are forced to be $0$.}, as $F$ and $G$ are contracted by the MMP.
Thus, we perform such a choice of coefficients, and we drop the dependence from ${\epsilon,\mu,\nu}$ in our notation.
Then, the generalized pair $(X^m,B^m+M^m)$ with data $X { \rar X^m \rar V}$ and $M$ satisfies the claimed conditions.
\end{proof}

In the case the input of Theorem \ref{generalized weak dlt model} is generalized log canonical, then the generalized pair $(X^m,B^m+M^m)$ is the crepant pullback of $(X',B'+M')$, and is therefore generalized log canonical.
Thus, we can talk about {\it{generalized dlt model}}, and we make this definition more precise with the following statement.

\begin{corollary}[{\cite[cf. 2.6.(2)]{Bir16a}}] \label{generalized dlt model}
Let $(X',B'+M')$ be a generalized pair with data $X \rar X' { \rar V}$ and $M$. 
Assume $(X',B'+M')$ is generalized log canonical.
Then $(X',B'+M')$ admits a $\Q$-factorial weak generalized dlt model $(X^m,B^m+M^m)$ such that $\K X'. + B' +M'=(f \sups m.)^*(\K X'. + B' + M')$.
We will call $(X^m,B^m+M^m)$ a {\it{generalized dlt model}} for $(X',B'+M')$.
\end{corollary}

In some situations, it is useful to extract certain divisors on a weak generalized dlt model.
The following proposition makes this precise \cite[cf. Lemma 4.5]{BZ}.

\begin{proposition} \label{weak generalized dlt model with extraction}
Let $(X',B'+M')$ be a generalized pair with data $X \rar X' {\rar V}$ and $M$.
Let $W' \subset X'$ be a generalized log canonical center of $(X',B'+M')$, and let $P$ be a generalized log canonical place with center $W'$.
Then, there exists a weak generalized dlt model $(X^m,B^m+M^m)$ such that $P$ is a divisor on $X^m$.
\end{proposition}

\begin{proof}
In case $W'$ is a divisor, the statement is trivial.
Therefore, we can assume that $P$ is exceptional over $X'$.
Without loss of generality, we may assume that $P$ appears on $X$, and that $f \colon X \rar X'$ is a log resolution of $(X',B')$.
Define $\Gamma' \coloneqq B' \wedge \Supp (B')$, and $\Gamma \coloneqq f_* \sups -1. \Gamma'$.
Denote by $E$ the exceptional divisor with reduced structure, and set $\Delta \coloneqq \Gamma + E$.

Then, we have
$$
\K X. + \Delta + M \sim \subs \R,f. A- C,
$$
where $A \geq 0$ is supported on the exceptional divisors with generalized discrepancy strictly greater than $-1$, and $C \geq 0$ is supported on the divisors with generalized discrepancy strictly less than $-1$.
Notice that $C$ may have components that are not exceptional over $X'$.
Now, run the $(\K X. + \Delta + M)$-MMP over $X'$ with scaling of an ample divisor \cite[p. 17]{BZ}.
After finitely many steps, we reach a model $X''$ such that $\K X''. + \Delta'' + M''$ is limit of divisors that are movable over $X'$.
Thus, it intersects non-negatively the very general curves over $X'$ of any divisor that is exceptional for $X'' \rar X'$.
Then, the same holds true for $A''-C''$. 
Therefore, by \cite[Lemma 3.3]{Bir12}, $A''=0$.
Hence, $X'' \rar X'$ extracts just divisors of generalized discrepancy at most $-1$.
Now, as $(X',B'+M')$ is generalized log canonical in a neighborhood $U'$ of the generic point of $W'$, over $U'$ we just performed the proof of \cite[Lemma 4.5]{BZ}.
That is, we extracted a prescribed set of divisors with negative generalized discrepancies.
In particular, we have not contracted $P$.

Now, let $(X'',B''+M'')$ be the trace of the generalized pair $(X',B'+M')$ on $X''$.
By construction, it is a generalized pair, as $B'' \geq 0$. Let $(X^m,B^m+M^m)$ be a weak generalized dlt model of $(X'',B''+M'')$. Then, $(X^m,B^m+M^m)$ satisfies the claimed properties.
\end{proof}

\section{Towards a generalized canonical bundle formula}

As first noticed by Kawamata \cite{Kaw98}, and then studied by Ambro \cite{Amb99}, the canonical bundle formula is needed to formulate an adjunction theory for higher codimensional log canonical centers.
Thus, in order to generalize the ideas developed in \cite{BZ} and \cite{Bir16a}, we need to extend the machinery of fiber space adjunction to generalized pairs.

Let $(X',B'+M')$ be a generalized sub-pair with data $X \rar  X'$ and $M$.
Let $f \colon  X' \rar Z'$ be a contraction where $\dim Z' > 0$.
Assume that $(X',B'+M')$ is generalized sub-log canonical near the generic fiber of $f$, and $\K X'. + B' +M' \sim_{\R,f} 0$.
For any prime divisor $D'$ on $Z'$, let $t_{D'}$ be the generalized log canonical threshold of $f^*D'$ with respect to $(X',B'+M')$ over the generic point of $D'$.
Then, set $B_{Z'} \coloneqq \sum b_{D'} D'$, where $b_{D'} \coloneqq 1 - t_{D'}$.
By construction, there is an $\R$-Cartier divisor $L_{Z'}$ such that $\K X'. + B' + M' \sim_\R f^* L_{Z'}$.
Define $M_{Z'} \coloneqq L_{Z'} - (\K Z'. + B_{Z'})$.
Hence, we can write
$$
\K X'. + B' + M' \sim_\R f^*(\K Z'. + B_{Z'} + M_{Z'}).
$$
We refer to this operation as {\it generalized adjunction for fiber spaces}.

\begin{remark} \label{remark definitions} {\em
As in the case of the usual adjunction for fiber spaces, $B_{Z'}$ is a well defined and uniquely determined divisor, while $M_{Z'}$ is defined up to $\R$-linear equivalence.}
\end{remark}

Now, let $\tilde X$ and $\tilde Z$ be higher birational models of $X'$ and $Z'$ respectively, and assume we have a commutative diagram of morphisms as follows

\begin{center}
\begin{tikzpicture}
\matrix(m)[matrix of math nodes,
row sep=2.6em, column sep=2.8em,
text height=1.5ex, text depth=0.25ex]
{\tilde X & X'\\
\tilde Z & Z'\\};
\path[->,font=\scriptsize,>=angle 90]
(m-1-1) edge node[auto] {$\phi$} (m-1-2)
edge node[auto] {$g$} (m-2-1)
(m-1-2) edge node[auto] {$f$} (m-2-2)
(m-2-1) edge node[auto] {$\psi$} (m-2-2);
\end{tikzpicture}
\end{center}

We denote by $\tilde M$ the trace of the moduli part on $\tilde X$. As usual, define $\tilde B$ via the identity
$$
\K \tilde X. + \tilde B + \tilde M = \phi^* (\K X'. + B' + M').
$$
Furthermore, set $L_{\tilde Z} \coloneqq \psi^* L_{Z'}$. With this piece of data, we can define divisors $B_{\tilde Z}$ and $M_{\tilde Z}$ such that
$$
\K \tilde X. + \tilde B + \tilde M \sim_\R g^*(\K \tilde Z. + B_{\tilde Z} + M_{\tilde Z}),
$$
$B_{Z'}= \psi_* B_{\tilde Z}$ and $M_{Z'}= \psi_* M_{\tilde Z}$.
In this way, Weil b-divisors $\mathbf{B} \subs Z'.$ and $\mathbf{M} \subs Z'.$ are defined.
We write $\bB Z',\tilde{Z}.$ and $\bM Z',\tilde{Z}.$ for the traces of $\bB Z'.$ and $\bM Z'.$ on any higher model $\tilde{Z}$.

Now, in the same fashion as the classic theory, we would like to establish properties of the b-divisors $\mathbf{B} \subs Z'.$ and $\mathbf{M} \subs Z'.$. 
Before doing so, we need to recall a few more technical ingredients.

Given an $\R$-Weil b-divisor $\mathbf{D}$ on $X$, we can define an associated {\it b-divisorial sheaf} $\O X. (\mathbf{D})$ as follows.
For every open set $U \subset X$, we define $\Gamma( U, \O X. (\mathbf{D}))$ as the set of rational functions $\alpha \in k(X)$ such that $\mult_E(\dv (\alpha) + \mathbf{D}) \geq 0$ for every valuation $E$ whose center satisfies $c_X(E) \cap U \neq \emptyset$. 

Recall that a b-divisor $\bf D$ is called {\it b-nef}$/S$ ({\it b-free}$/S$, {\it b-semi-ample}$/S$, {\it b-big}$/S$) if there exists a birational morphism $X \rar X'$ such that ${\bf D} = \overline{{\bf D} \subs X.}$, and ${\bf D} _X$ is nef (free, semi-ample, big) relatively to the morphism $X \rar S$.

Let $(X',B'+M')$ be a generalized sub-pair with data $X \rar  X'$ and $M$.
We denote by $\mathbf{K} \subs X'.$ and $\mathbf{M} \subs X'.$ the canonical b-divisor of $X'$ and the moduli b-divisor respectively.
We define the {\it generalized discrepancy b-divisor} as
$$
\mathbf{A}(X',B'+M') \coloneqq \mathbf{K} \subs X'. + \mathbf{M} \subs X'. - \overline{\K X'. +B' +M'},
$$
where the overline symbol denotes the $\R$-Cartier b-divisor associated to an $\R$-Cartier divisor.
We will write just $\mathbf{A}$ if there will be no ambiguity.
We also set
$$
\mathbf{A}^*(X',B'+M') \coloneqq \mathbf{A}(X',B'+M') + \sum \subs a_E(X',B'+M')=-1. E.
$$

\begin{remark} \label{remark moduli descends} {\em
If $X''$ is a model where $M$ descends, then the b-divisors $\mathbf{A}(X',B'+M')$ and $\mathbf{A}^*(X',B'+M')$ agree with the usual b-divisors $\mathbf{A}(X'',\mathbf{B} \subs X''.)$ and $\mathbf{A}^*(X'',\mathbf{B} \subs X''.)$ on all the models $X'''$ over $X''$ \cite[cf. pp. 5-6]{Amb04}.}
\end{remark}

As explained in \cite[Remark 2.2]{Amb04} and \cite[Definition 8.4.2]{K07}, the b-divisors $\mathbf{A}$ and $\mathbf{A}^*$ are important to use Hodge theoretic techniques to investigate b-nefness and b-semi-ampleness of the moduli b-divisor $\bM X'.$.
In particular, they impose certain conditions that guarantee that a specific vector bundle is a line bundle.
For similar reasons, such conditions are needed in this work.

\begin{remark} \label{base change property} {\em
As the moduli part $\bM X.$ descends to some model $X$, the boundary part $\bB Z'.$ satisfies similar properties as in the classical theory.
In particular, the finite base change property still holds \cite[Theorem 3.2]{Amb99}.

More precisely, let $\alpha  \colon  Y' \rar X'$ be a generically finite map from a normal projective variety $Y'$.
Also, let $Y$ be a higher model of $Y'$ admitting a morphism $\beta \colon  Y \rar X$.
Denote by $\varphi$ the morphism $\varphi \colon  Y \rar Y'$. Let $\K Y. + B_Y \coloneqq \beta^*(\K X. + B)$ be the crepant pullback of $\K X. + B$ to $Y$.
Also, denote $M_Y \coloneqq \beta^*M$.
Then, define $B \subs Y'.\coloneqq \varphi_* B_Y$ and $M \subs Y'. \coloneqq \varphi_* M_Y$.
Thus, we induced a structure of generalized sub-pair on $Y'$.

Now, let $\gamma  \colon  W' \rar Z'$ be the normalization of $Z'$ in $Y'$.
Denote by $B \subs W'.$ and $M \subs W'.$ the boundary and moduli parts induced on $W'$ by $Y' \rar W'$.
Then, we have $\K W'. + B \subs W'. = \gamma^*(\K Z'. + B \subs Z'.)$, and $M \subs W'. = \gamma^* M \subs Z'.$.

In particular, in order to prove that the moduli b-divisor $\bM Z'.$ induced by a fibration $X' \rar Z'$ is a Cartier b-divisor, we are free to replace $X' \rar Z'$ with a fibration induced by generically finite base change.
Furthermore, the same reduction applies when we want to show b-nefness or b-semi-ampleness \cite[Example 1.4.4.(ii)]{LAZ1}, \cite[Theorem 1.20]{Fuj83}.}
\end{remark}

Now, we recall two natural constructions that arise in view of Remark \ref{base change property}.

\begin{defn}{\em
Let $f \colon  X \rar Z$ be a contraction, and let $(X,B)$ be a sub-pair.
We say that the morphism $f$ is \emph{prepared} if the following properties are satisfied:
\begin{itemize}
\item $X$ and $Z$ are smooth;
\item there is a simple normal crossing divisor $\Sigma \subset Z$ such that $g  \colon  X \rar Z$ is smooth over $Z \setminus \Sigma$;
\item $\Supp(B)+g^*\Sigma$ has simple normal crossing support; and
\item $B$ is relatively simple normal crossing over $Z \setminus \Sigma$.
\end{itemize}
Equivalently, we call the above properties \emph{standard normal crossing assumptions} \cite[Definition 8.3.6]{K07}.}
\end{defn}

Now, we will make use of some constructions related to toric and toroidal geometry.
We refer to \cite{Kar99,AK} for the key definitions and properties.
We will recall just the facts that we will explicitly use.

\begin{defn}{\em
Let $f \colon  X \rar Y$ be a toroidal morphism.
We say that $X$ has \emph{good horizontal divisors} if at every point $x \in X$ we can find a local model of the form
$$
X_\sigma = X \subs \sigma'. \times \af l.,
$$
where the horizontal divisors in $X \setminus U_X$ through $x$ are exactly the pullbacks of the coordinate hyperplanes in $\af l.$.}
\end{defn}

\begin{defn}{\em
A toroidal morphism $f  \colon  X \rar Y$ with good horizontal divisors is called \emph{weakly semi-stable} if
\begin{itemize}
\item the morphism $f$ is equidimensional;
\item all the fibers of $f$ are reduced; and
\item $Y$ is non-singular.
\end{itemize}
If also $X$ is non-singular, we say that the morphism $f  \colon  X \rar Y$ is \emph{semi-stable}.}
\end{defn}

Now, we include a technical statement that will be useful in the following.

\begin{proposition} \label{prop technical}
Let $f \colon  X \rar Z$ be a morphism of projective varieties, $M$ be an $\R$-Cartier divisor on $X$, and $H$ be an ample divisor on $Z$.
If $M$ is nef on $X$ and relatively semi-ample over $Z$, then $M+\epsilon f^*H$ is semi-ample for any $\epsilon > 0$.
\end{proposition}

\begin{proof}
First, assume that $M$ is relatively ample over $Z$.
Then, for $l \gg 1$, $M+lf^*H$ is ample, as $H$ is ample on $Z$, and $M$ ample over $Z$. 
Then, as $M$ is nef, $kM+(M+lf^*H)$ is ample for any $k>0$.
As we can choose the real numbers $l$ and $k$ so that the ratio $\frac{l}{k+1}$ equals any given $\epsilon > 0$, the claim follows.

Now, consider the general setup of the statement.
As $M$ is relatively samiample, there exists a morphism $g \colon  X \rar Y$ over $Z$, such that $M \sim \subs \R,Z. g^*N$, where $N$ is ample over $Z$ \cite[p. 22]{HK10}.
Without loss of generality, we may assume that $g$ is surjective.
Also, up to twisting $N$ by the pullback of an $\R$-Cartier divisor on $Z$, we may assume that $M \sim \subs \R. g^*N$.
Then, by \cite[Example 1.4.4.(ii)]{LAZ1}, $N$ is also nef.
Hence, by the previous step, $N+ \epsilon h^*H$ is semi-ample for any $\epsilon > 0$, where $h$ denotes the morphism $h  \colon  Y \rar Z$.
As $M + \epsilon f^* H \sim \subs \R. g^*(N+ \epsilon h^*H)$, the claim follows.
\end{proof}

Before proving Theorem \ref{lemma klt fibrations}, we need to introduce a new description of the construction of boundary and moduli parts in the setup of generalized pairs.

\begin{remark} \label{remark min max} {\em
Let $(X',B'+M')$ be a generalized sub-pair with data $X \rar X'$ and $M$.
Let $f \colon  X' \rar Z'$ be contraction, where $\dim Z' > 0$.
Assume that $(X',B'+M')$ is generalized sub-log canonical near the generic fiber of $f$, and that $\K X'. + B' + M' \sim \subs \R,f. 0$.
Then, we can define b-divisors $\bB Z'.$ and $\bM Z'.$ on $Z'$.
In order to do so, we are free to replace $X'$ and $Z'$ with higher models $X''$ and $Z$.
We may assume that the morphism $g \colon  X'' \rar Z$ is prepared, and that $M$ descends onto $X$.
Denote by $\Sigma \subset Z $ the simple normal crossing divisor as in the definition of prepared morphism.
For notation's sake, we may write $X''=X$. 

Now, assume $M$ is relatively semi-ample over $Z$, and that $\rk f_* \O X'.(\lceil \bA.^*(X',B'+M') \rceil)=1$.
Furthermore, let $M$ and $B$ be $\Q$-divisors.
Fix an ample $\Q$-divisor $H$ on $Z$. As $M$ is relatively semi-ample, by Proposition \ref{prop technical}, for any rational number $\epsilon > 0$, the $\Q$-linear series $|M+ \epsilon g^* H|_\Q$ is basepoint-free.
Then, the generalized pair $(X,B+(M+\epsilon g^* H))$ with data $id \colon X \rar  X$ satisfies the same properties as $(X,B+M)$.
Notice that we have not changed the boundary part, which is still $B$, while we have perturbed the moduli part, which is $M+\epsilon g^*H$.

By definition, the nef part of a generalized pair does not contribute to the singularities once it descends.
Thus, as both $M$ and $M + \epsilon g^* M$ descend to $X$, $(X,B+M)$ and $(X,B+(M+\epsilon g^* H))$ have the same generalized discrepancies.
Therefore, the boundary b-divisors that they induce over $Z'$, denoted by $\bB Z'.$ and $\bB Z'.^\epsilon$ respectively, are equal.
The moduli b-divisors are related by the identity
$$
\bM Z'.^\epsilon = \bM Z'. + \epsilon \overline{H}.
$$
Therefore, we have
\begin{equation} \label{equation limit bM}
\bM Z'. = \lim \subs \epsilon \to 0. \bM Z'. ^\epsilon.
\end{equation}

As $\bM Z'.$ being a $\Q$-Cartier b-divisor is equivalent to $\bK Z'. + \bB Z'.$ having the same property, we can investigate this aspect through $(X,B+(M+ \epsilon g^*H))$.
Henceforth, unless otherwise stated, we fix a rational number $\epsilon > 0$.

As $|M+\epsilon g^* H|_\Q$ is basepoint-free, we can take a general element $0 \leq A \sim_\Q M+\epsilon g^* H$ such that $\Supp A$ is smooth, $B + g^* \Sigma + A$ has simple normal crossing support, and $(X,B+A)$ is sub-log canonical over the generic point of $Z$. Also, we have $\rk g_* \O X.(\lceil \bA.^*(X,B+A) \rceil)=1$.
By the classic theory of adjunction for fiber spaces \cite{Amb04,K07,FG14}, it induces b-divisors $\bB Z'.^A$ and $\bM Z'.^A$ on $Z'$.
By construction, $\bB Z'.^A \geq \bB Z'. \sups \epsilon.$ and $\bM Z'.^A \leq \bM Z'.\sups \epsilon.$.
Furthermore, since we may assume that the coefficients of $A$ are small enough, the fact that the divisor $B + g^* \Sigma + A$ has simple normal crossing support implies that the multiplicities of $\bB Z'.^A$ and $\bB Z'. \sups \epsilon.$ are the same along the prime divisors in $\Sigma$.

Now, fix a prime divisor $P \subset Z$ that is not supported on $\Sigma$.
We can assume that $A$ meets $g^{-1}(P)$ transversally.
Thus, $g$ is prepared for $(X,B+A)$ over a neighborhood of $\eta_P$.
Furthermore, as $A$ is horizontal over $Z$, we have that $\bB Z',Z.^A = \bB Z',Z.\sups \epsilon.$ along $\eta_P$.
As for any choice of $A$ the difference $\bB Z',Z.\sups \epsilon. - \bB Z',Z.^A$ is supported on finitely many prime divisors, we may find $A_1, \ldots , A_l \in |M+\epsilon g^* H|_\Q$ such that $\bB Z',Z.\sups \epsilon. = \min _{1 \leq i \leq l} \bB Z',Z.\sups A_i.$.
This is equivalent to $\bM Z',Z.\sups \epsilon. = \max _{1 \leq i \leq l} \bM Z',Z.\sups A_i.$.

While for a fixed model $Z$ we can recover $\bB Z',Z.^\epsilon$ and $\bM Z',Z.^\epsilon$ with finitely many choices of $0 \leq A \sim_\Q M+\epsilon g^*H$, in general, we need infinitely many to recover the whole b-divisors.
In particular, we have
$$
\bB Z'.^\epsilon =\inf \subs 0 \leq A \sim_\Q M + \epsilon g^*H. \bB Z'.^A, \quad \bM Z'.^\epsilon =\sup \subs 0 \leq A \sim_\Q M + \epsilon g^*H. \bM Z'.^A.
$$
Notice that, although the traces of $\bM Z'.$, $\bM Z'.^\epsilon$ and $\bM Z'.^A$ are well defined just up to $\Q$-linear equivalence, we can treat those as honest divisors once we fix representatives of the classes $\K Z'. + B \subs Z'. + M \subs Z'.$ and $H$.
Furthermore, as $M+\epsilon g^* H$ is semi-ample, we can restrict the infimum to the $A$'s such that $(X,B+A)$ is sub-log canonical over the generic point of $Z'$.
Call such class $\Xi$.

Notice that, if $M$ is semi-ample, we can take $\epsilon=0$ in the above discussion.

In general, we are interested in proving that $\bM Z'.$ is a b-nef $\Q$-Cartier b-divisor.
As $\bB Z'. = \bB Z'.^\epsilon$, we can reduce the analysis to the case when $M$ is semi-ample on $X$.
By the same argument, if we know that $\mathbf{K}_{Z'} + \bB Z'.$ is $\Q$-Cartier, $\bM Z'.$ and all $\bM Z'. ^\epsilon$ descend to the same model.
Thus, in virtue of equation \eqref{equation limit bM}, $\bM Z'.$ is b-nef if so is $\bM Z'. ^\epsilon$ for any $\epsilon > 0$.
Hence, also b-nefness can be reduced to the case when $M$ is semi-ample.}
\end{remark}

Now, we need to introduce some terminology.

\begin{defn}{\em
Let $\bD.$ be a b-divisor over $X$.
We say that $\bD.$ is \emph{almost b-nef} if the following holds: for every higher models $X'$ and $X''$ of $X$ where the traces of $\bD.$ are $\R$-Cartier, with morphism $f \colon  X'' \rar X'$, we have $\bD X''. \leq f^* \bD X'.$.}
\end{defn}

The perspective in Remark \ref{remark min max} allows us to prove the following key statement.

\begin{proposition} \label{lemma almost b-nef}
Let $(X',B'+M')$ be a {projective} generalized sub-pair with data $X \rar X'$ and $M$.
Assume that $B'$, $M'$ and $M$ are $\Q$-divisors. Let $f \colon X' \rar Z'$ be a contraction such that $\K X'. + B' + M' \sim_{\Q,f} 0$ and $M$ is semi-ample. Also, let $(X',B'+M')$ be generalized sub-log canonical over the generic point of $Z'$, with $\rk f_* \O X'. (\lceil \bA.^*(X',B'+M') \rceil)=1$. Then, the b-divisor $\bM Z'.$ is almost b-nef.
\end{proposition}

\begin{proof}
Fix two smooth models $\hat Z$ and $\tilde{Z}$ of $Z'$, and assume we have a morphism $\phi  \colon  \hat Z \rar \tilde Z$. Also, let $\Xi$ be as in Remark \ref{remark min max}. By the classic theory of the canonical bundle formula and the negativity lemma \cite{Amb04,FG14}, for every $A \in \Xi$, we have $\bM Z',\hat Z.^A \leq \phi^* \bM Z',\tilde{Z}.^A$. Then, by Remark \ref{remark min max}, we have
\begin{equation*}
\begin{split}
\bM Z',\hat Z. &= \sup \subs  A \in \Xi. \bM Z',\hat Z.^A\\
&\leq \sup \subs  A \in \Xi. \phi^* \bM Z',\tilde Z.^A\\
& \leq \phi^* \sup \subs  A \in \Xi. \bM Z',\tilde Z.^A\\
&= \phi^* \bM Z',\tilde{Z}..
\end{split}
\end{equation*}
This proves the claim.
\end{proof}

\begin{remark} \label{remark rel dim 1 M almost nef and good} {\em
In case $(X',B'+M')$ is generalized klt over the generic point of $Z'$, by \cite[Theorem 3.3]{Amb05}, each $\bM Z'.^A$ is b-nef and b-good.
In particular, we have that $\bM Z'.$ dominates a b-nef and b-good divisor.}
\end{remark}

Now, we are ready to address the first result towards a generalized canonical bundle formula.

\begin{theorem} \label{lemma klt fibrations}
Let $(X',B'+M')$ be a {projective} generalized sub-pair with data $X \rar X'$ and $M$.
Assume that $B'$, $M'$ and $M$ are $\Q$-divisors.
Let $f  \colon  X' \rar Z'$ be a contraction such that $\K X'. + B' + M' \sim \subs \Q, f. 0$ and $M$ is relatively semi-ample over $Z'$.
Also, let $(X',B'+M')$ be generalized sub-log canonical over the generic point of $Z'$, with $\rk f_* \O X'. (\lceil \bA.^*(X',B'+M') \rceil)=1$.
Then, the b-divisor $\bM Z'.$ is $\Q$-Cartier and b-nef.
\end{theorem}

For the reader's convenience, we will split it into two statements.

\begin{theorem} \label{bM b-divisor}
Let $(X',B'+M')$ be a {projective} generalized sub-pair with data $X \rar X'$ and $M$.
Assume that $B'$, $M'$ and $M$ are $\Q$-divisors.
Let $f \colon X' \rar Z'$ be a contraction such that $\K X'. + B' + M' \sim_{\Q,f} 0$ and $M$ is relatively semi-ample over $Z'$.
Also, let $(X',B'+M')$ be generalized sub-log canonical over the generic point of $Z'$, with $\rk f_* \O X'. (\lceil \bA.^*(X',B'+M') \rceil)=1$.
Then, the b-divisor $\bM Z'.$ is $\Q$-Cartier.
In particular, if $(X,B) \rar Z$ is weakly semi-stable with good horizontal divisors, $\bM Z'.$ descends onto $Z$.
\end{theorem}

\begin{proof}
By Remark \ref{base change property} and \cite[Theorem 9.5]{Kar99}, we may assume that {$Z$ is projective and} $g \colon  (X,B) \rar Z$ is weakly semi-stable with good horizontal divisors.
Let $\Sigma \coloneqq Z \setminus U_Z$ be the toroidal divisor on the base.
Up to adding to $B$ the pullback of a divisor supported on $\Sigma$, we may assume $\Sigma = \bB Z',Z.$.
Also, by the discussion in Remark \ref{remark min max}, we may assume that $M$ is semi-ample.

Notice that, by weak semi-stability, all the fibers are reduced and semi-log canonical \cite[p. 90]{Kar99}.
Indeed, as $X$ is Gorenstein, so are the fibers.
In particular, they are $S_2$.
Then, the constraint on codimension 1 singularities is local in nature \cite[Chapter 3]{HK10}.
More precisely, it can be checked after completion of the local rings.
Similarly, the computation of discrepancies is local in nature as well \cite[cf. Remark 4.6]{DFDC}.

Now, by \cite[Lemma 3.1]{Flo14}, for the computation of $\lct \subs \eta_P. (X,0;g^*P)$ we may assume that $Z$ is a curve.
Therefore, by inversion of adjunction, we have that $\lct \subs \eta_P. (X,0;g^*P)=1$ for every prime divisor $P \subset Z$.
Furthermore, by the assumption of $B^h$ being good, we have that $B^h$ does not contribute to the computation of $\bB Z',Z.$.
Indeed, locally, a fiber $(X_z,\Supp(B^h)_z)$ can be thought as $Y_z \times (\af l.,D)$, where $Y_z$ is semi-log canonical and $D=\sum \subs i=1. ^l D_i$ is the union of the coordinate hyperplanes in $\af l.$.
By induction on $l$, $Y_z \times (D_1,\sum \subs i=2.^l D_i \cap D_1)$ is semi-log canonical, and by inversion of adjunction so is $Y_z \times (\af l.,D)$.
Therefore, we have that $(X_z,\Supp(B^h)_z)$ is semi-log canonical.
As $B^h \leq \Supp(B^h)$, by inversion of adjunction we have $\lct \subs \eta_P. (X,B^h;g^*P)=1$.
Therefore, we conclude $B^v \leq g^* \Sigma$.

Let $\pi  \colon  Z'' \rar Z$ be a birational morphism such that $Z''$ is smooth, and $\pi \sups -1. (\Sigma)$ is simple normal crossing.
Also, we define $X'' \coloneqq X \times_Z Z''$.
Then, by \cite[Lemma 8.3]{Kar99} and the discussion \cite[p. 59]{Kar99}, the morphism $h  \colon  (X'',B'') \rar Z''$ is weakly semi-stable with good horizontal divisors.

By the above arguments, it follows that $B^v \leq g^* \Sigma$.
Hence, we have the inequality
$$
\K X. + B \leq \K X/Z. + B^h + g^*(\K Z. + \Sigma).
$$
Considering the pullback via $\phi  \colon  X'' \rar X$, we obtain
\begin{equation} \label{inequality log boundaries}
\K X''. + B''  \leq \K X''/Z''. + (B'')^h + h^*(\K Z''. + \Sigma''),
\end{equation}
where $\Sigma''$ denotes the log-pullback of $\Sigma$ to $Z''$.
Notice that, by the geometric assumptions, $(B'')^h=\phi^*B^h$.
Our goal is to show $\Sigma''=\bB Z',Z''.$.
By Proposition \ref{lemma almost b-nef}, we have $\bB Z',Z''. \geq \Sigma''$.
Notice that $\bB Z',Z''.$ is computed via the singularities of $\K X''. + B''$.
By inequality \eqref{inequality log boundaries}, $\bB Z',Z''. \leq \Gamma''$, where $\Gamma''$ is the boundary on $Z''$ induced by the singularities of $(X'',(B'')^h+h^*(\Sigma''))$.
As $h  \colon  (X'',(B'')^h) \rar Z''$ is weakly semi-stable with good horizontal divisors, we have $\Gamma''=\Sigma''$.
This concludes the proof.
\end{proof}

\begin{remark} \label{remark surface base} {\em
Recall that the pushforward of a nef divisor under a birational morphism of normal surfaces is nef.
Furthermore, on a normal projective surface $S$, the maximum $M$ of finitely many nef divisors $M_1, \ldots, M_k$ is nef.
Indeed, fix an irreducible curve $C \subset S$.
Then, fix $i \in \lbrace 1,\ldots , k \rbrace$ such that $\mult \subs C. M = \mult \subs C. M_i$.
Then, we can write $M = M_i + E$, where $E \geq 0$ and $C \not \subset \Supp(E)$.
Then, we have $M \cdot C = M_i \cdot C + E \cdot C \geq M_i \cdot C \geq 0$.

Then, by Remark \ref{remark min max}, it follows that, in the setup of Theorem \ref{bM b-divisor}, if $Z'$ is a surface, $\bM Z',Z.$ is nef for every model $Z \rar Z'$.}
\end{remark}

\begin{theorem} \label{bM b-nef}
Let $(X',B'+M')$ be a {projective} generalized sub-pair with data $X \rar X'$ and $M$.
Assume that $B'$, $M'$ and $M$ are $\Q$-divisors.
Let $f \colon X' \rar Z'$ be a contraction such that $\K X'. + B' + M' \sim_{\Q,f} 0$ and $M$ is relatively semi-ample over $Z'$.
Also, let $(X',B'+M')$ be generalized sub-log canonical over the generic point of $Z'$, with $\rk f_* \O X'. (\lceil \bA.^*(X',B'+M') \rceil)=1$.
Then, the b-divisor $\bM Z'.$ is b-nef.
\end{theorem}

\begin{proof}
By Theorem \ref{bM b-divisor}, we know that $\bM Z'.$ is a $\Q$-Cartier b-divisor.
Assume by contradiction that $\bM Z'.$ is not b-nef.
Then, for every model $Z$ where $\bM Z'.$ descends, there is a curve $C \subset Z$ such that $\bM Z',Z. \cdot C < 0$.

Without loss of generality, we may assume that $Z$ is smooth.
Furthermore, by the projection formula for cycles and by blowing up the singular points of $C$, we may assume that $C$ is smooth.

Now, let $\pi \colon  Z'' \rar Z$ be the blow-up of $Z$ along $C$.
By the projection formula, every curve $C'' \subset \pi \sups -1. (C)$ that dominates $C$ is such that $\bM Z',Z''. \cdot C'' < 0$.
Let $S'' \subset Z''$ be a smooth surface obtained by general hyperplane cuts.
Then, by \cite[Lemma 3.1]{Flo14}, we have $\bM Z',Z''.| \subs S''.= \bM S'',S''.$, where $\bM S''.$ is the moduli b-divisor of the induced fibration with base $S''$.
By the positivity of $S''$, there exists $C''$ as above with $C'' \subset S''$.
On the other hand, by Remark \ref{remark surface base}, $\bM S'',S''.$ is nef.
Thus, we get a contradiction, and the claim follows.
\end{proof}

We conclude considering the relation between the singularities of the source of the fibration and the ones of the generalized pair induced on the base.
In doing so, we follow ideas of Ambro \cite[Proposition 3.4]{Amb99}, \cite[Theorem 3.1]{Amb04}.

\begin{proposition} \label{inv adj fiber spaces}
Let $(X',B'+M')$ be a generalized sub-pair with data $X \rar X'$ and $M$.
Let $f \colon  X' \rar Z'$ be a contraction such that $\K X'. + B' + M' \sim \subs \R,f. 0$ and $(X',B'+M')$ is generalized sub-log canonical over the generic point of $Z'$.
Assume that a generalized pair structure $(Z',\bB Z'. + \bM Z'.)$ is induced on $Z'$.
Furthermore, let $g \colon  X \rar Z$ be a birational model of $f \colon X' \rar Z'$ such that $\bM Z'.$ descends to $Z$, and $M$ descends to $X$.
Then, $(Z,\bB Z',Z.)$ is sub-log canonical in a neighborhood of $z \in Z$ if and only if $(X,B)$ is sub-log canonical in a neighborhood of $g \sups -1. (z)$.
Furthermore, if $(X',B'+M')$ is generalized klt over the generic point of $Z'$, $(Z,\bB Z',Z.)$ is sub-klt in a neighborhood of $z \in Z$ if and only if $(X,B)$ is sub-klt in a neighborhood of $g \sups -1. (z)$.
\end{proposition}

\begin{proof}
The proof of \cite[Proposition 3.4]{Amb99} goes through verbatim.
\end{proof}

\begin{remark}{\em
As $\bB X'.$ and $\bB Z'.$ both descend to $X$ and $Z$ respectively, in the statement of Proposition \ref{inv adj fiber spaces}, we can equivalently replace sub-klt and sub-log canonical with their generalized versions.}
\end{remark}

\section{The case of effective boundary}

In this section, under the assumption that $B'$ is effective over the generic point of $Z'$, we weaken some conditions of Theorem \ref{lemma klt fibrations}.
The constraint on the horizontal part of $B'$ is due to results used in the proofs, such as \cite[Theorem 0.1]{Amb05} and the MMP \cite{BCHM}.
We start with a technical statement.

\begin{lemma} \label{lemma diagram of fibrations}
Let $(X',B' \subs X'.+M' \subs X'.)$ be a generalized pair with data $X$ and $M_X$.
Let $f \colon  X' \rar Z'$ be a contraction such that $\K X'. + B' \subs X'. + M' \subs X'. \sim \subs \R,f. 0$.
Also let $Y'$ be a variety such that there exist contractions $g \colon  X' \rar Y'$ and $h \colon  Y' \rar Z'$ satisfying $f=h \circ g$.
Let $\bB Z'.$ and $\bB Y'.$ be the boundary b-divisors induced on $Z'$ and $Y'$ by $f$ and $g$ respectively.
Also, denote by $\bD Z'.$ the boundary b-divisor induced on $Z'$ by $h$ and $\K Y'. + \bB Y',Y'. + \bM Y',Y'.$.
Assume that $\bM Y'.$ is a Cartier b-divisor.
Then, we have $\bB Z'. = \bD Z'.$.
\end{lemma}

\begin{proof}
Since we are comparing b-divisors, we are free to replace each variety with a higher model.
Thus, we can replace $X'$, $Y'$ and $Z'$ with models such that:
\begin{itemize}
\item the moduli b-divisors $\bM X'.$ and $\bM Y'.$ descend onto $X$ and $Y$ respectively.
We will denote their traces by $M_X$ and $M_Y$;
\item the morphisms $X \rar Z$, $X \rar Y$ and $Y \rar Z$ are all prepared.
By abusing notation, we will still denote those as $f$, $g$ and $h$ respectively.
\end{itemize}
For ease of notation, we will write $B_X = \bB X',X.$, $B_Y= \bB Y',Y.$, $B_Z= \bB Z',Z.$ and $D_Z= \bD Z',Z.$.
Let $P$ be a prime divisor in $Z$.
We have to compare $\mult_P B_Z$ and $\mult_P D_Z$.

Let $R_1, \ldots, R_k$ be the prime divisors in $X$ that dominate $P$.
Similarly, denote by $Q_1, \ldots, Q_l$ the prime divisors in $Y$ dominating $P$.
Notice that $g(R_i)$ is not necessarily a divisor; in case it is, we have $g(R_i)=Q \subs j(i).$ for some $1 \leq j(i) \leq l$.

By \cite[Remark 3.1.4]{Amb99}, the components of $B_X$ that dominate $Z$ do not contribute to the computations.
Thus, we may assume $B_X= \sum \subs i=1. ^k b_i R_i$ over $\eta_P$, the generic point of $P$.
We can write $f^*P = \sum \subs i=1. ^k p^i R_i$, $g^* Q_j = \sum \subs i=1.  ^k q^i_j R_i$, and $h^*P= \sum \subs j=1.^l r^j Q_j$.
Since $f^*=g^* \circ h^*$, we get $p^i= \sum \subs j=1.^l r^j q^i_j$.

For the formula used in the following computations, we refer to \cite[Remark 3.1.4]{Amb99}.
We have
$$
\mult_P B_Z = \max_i \frac{b_i+p^i-1}{p^i}=\max_i \frac{b_i+\sum  r^j q^i_j-1}{\sum  r^j q^i_j},
$$
and
$$
\mult \subs Q_j. B_Y = \max \subs i|g(R_i)=Q_j. \frac{b_i+q^i_j-1}{q^i_j}.
$$
This implies the following formula
$$
\mult_P D_Z = \max_j \frac{\mult \subs Q_j.B_Y+r^j-1}{r^j}=\max_j \max \subs i|g(R_i)=Q_j. \frac{b_i+r^j q^i_j-1}{r^j q^i_j}. 
$$

Now, we may assume that $\mult_P B_Z$ is computed by $R_1$, and that $g(R_1)=Q_1$.
Thus, we have $q^1_j = 0$ if $j \neq 1$.
Therefore, we have
$$
\mult_P B_Z = \frac{b_1+r^1 q^1_1-1}{r^1 q^1_1} \leq \max_j \max \subs i|g(R_i)=Q_j. \frac{b_i+r^j q^i_j-1}{r^j q^i_j} = \mult_P D_Z.
$$
As already mentioned, if $g(R_i)$ is a divisor, then $q^i_j = 0$ if $j \neq j(i)$.
Thus, we have
\begin{equation*}
\begin{split}
\mult_P D_Z &= \max_j \max \subs i|g(R_i)=Q_j. \frac{b_i+r^j q^i_j-1}{r^j q^i_j} = \max \subs i|g(R_i) \; \mathrm{divisor}. \frac{b_i+r \sups j(i). q^i \subs j(i).-1}{r \sups j(i). q^i \subs j(i).}\\
&= \max \subs i|g(R_i) \; \mathrm{divisor}. \frac{b_i+\sum r \sups j. q^i \subs j.-1}{\sum r \sups j. q^i \subs j.} \leq \max_i \frac{b_i+\sum  r^j q^i_j-1}{\sum  r^j q^i_j} = \mult_P B_Z.
\end{split}
\end{equation*}
Hence, as $\mult_P D_Z = \mult_P B_Z$ and $P$ is arbitrary, we conclude that $\bB Z'. = \bD Z'.$.
\end{proof}

Before proving Theorem \ref{full generality theorem firbations}, we deal with a particular case of it.

\begin{lemma} \label{lemma MFS}
Let $X'$ be a {projective} $\Q$-factorial klt variety, and let $(X',B'+M')$ be a generalized sub-pair with data $X \rar X'$ and $M$.
Assume that $B'$, $M'$ and $M$ are $\Q$-divisors.
Let $f  \colon  X' \rar Z'$ be a contraction {to a projective variety $Z'$} such that $\K X'. + B' + M' \sim \subs \Q, f. 0$, $\rho(X'/Z') =1$ and $M'$ is relatively ample.
Also, let $(X',B'+M')$ be generalized log canonical over the generic point of $Z'$.
Then, the b-divisor $\bM Z'.$ is $\Q$-Cartier and b-nef.
\end{lemma}

\begin{proof}
Since $X'$ is $\Q$-factorial and klt, and $(X',B'+M')$ is generalized log canonical over $\eta \subs Z'.$, for any rational number $0 < \epsilon \ll 1$ the generalized pair $(X',(1-\epsilon)(B'+M'))$ with data $X$ and $(1-\epsilon)M$ is generalized klt over $\eta \subs Z'.$.
Since $\rho(X'/Z')=1$, $(B')^h$ is either $0$ or relatively ample.
Let $H$ be an ample divisor on $Z'$ such that $f^*H + (B')^h + M'$ is ample.
Also, write $\pi  \colon  X \rar X'$, and fix an effective and $\pi$-exceptional divisor $E$ such that $-E$ is $\pi$-ample.
Finally, we have $\pi^* M' = M - F$, where $F \geq 0$ is $\pi$-exceptional.

Let $B_\epsilon$ be defined by the identity
$$
\K X. + B_\epsilon + (1- \epsilon)M = \pi^*(\K X'. + (1-\epsilon)(B')^h + (B')^v + (1-\epsilon) M'),
$$
and set $B \coloneqq B_0$.
Then, for rational numbers $0 < \delta \ll \epsilon \ll 1$, we have
$$
\K X. + B + M \sim \subs \Q,Z'. \K X. + (B_\epsilon + \delta E - \epsilon F) + ((1-\epsilon)M + \epsilon \pi^* (M' + (B')^h + f^* H) - \delta E).
$$
Since $(X',(1-\epsilon)(B'+M'))$ is generalized klt over $\eta \subs Z'.$ and $0 < \delta = \delta(\epsilon) \ll \epsilon$, the sub-pair $(X,B_\epsilon + \delta E - \epsilon F)$ is sub-klt over $\eta \subs Z'.$.
Furthermore, as $\pi^* (M' + B' + f^* H) - \delta E)$ is ample by construction, then so is $M + \epsilon \pi^* (M' + B' + f^* H) - \delta E)$.
Finally, we have that $\pi_*(B_\epsilon + \delta E - \epsilon F)$ is effective over $\eta \subs Z'.$.

Therefore, the generalized sub-pair 
$$
(X,(B_\epsilon + \delta E - \epsilon F) + ((1-\epsilon)M + \epsilon \pi^* (M' + B' + f^* H) - \delta E))
$$
satisfies the hypotheses of Theorem \ref{lemma klt fibrations}.
This provides b-divisors $\bB Z'.^\epsilon$ and $\bM Z'. ^\epsilon$.
As $\Supp (B) \cup \Supp(E) \cup \Supp(F)$ is independent of $\epsilon$, by the proof of Theorem \ref{lemma klt fibrations}, the b-divisors $\bM Z'. ^\epsilon$ descend to the same higher model of $Z'$ for all $0 < \epsilon \ll 1$.
Thus, as the b-divisors $\bB Z'.$ and $\bM Z'.$ induced by $(X,B+M)$ satisfy
$$
\bB Z'. = \lim \subs \epsilon \to 0. \bB Z'.^\epsilon, \quad \bM Z'. = \lim \subs \epsilon \to 0. \bM Z'.^\epsilon,
$$
we conclude that $\bM Z'.$ is $\Q$-Cartier and b-nef.
\end{proof}

Now, we are ready to prove Theorem \ref{full generality theorem firbations}.

\begin{proof}[Proof of Theorem \ref{full generality theorem firbations}]
We will prove the statement by induction on the relative dimension of the fibration.
{Since $Z'$ is proper, by Remark \ref{base change property} and Chow's lemma, we may assume that $Z'$ is projective.}
By Theorem \ref{generalized weak dlt model}, we may assume that $X'$ is $\Q$-factorial, and that $(X',(B')^h)$ is dlt.
Notice that Theorem \ref{generalized weak dlt model} applies, as we can pull back divisors from $Z'$ to guarantee that $B'$ is effective.
Let $X' \subs \overline{\eta}_{Z'}.$ be the geometric generic fiber.
Then, by assumption, $(X' \subs \overline{\eta}_{Z'}.,B'\subs \overline{\eta}_{Z'}.)$ is dlt.
Notice that, as $B'$ is effective over the generic point of $Z'$, the technical assumptions regarding $\rk f_* \O X'. (\lceil \bA.^*(X',B'+M') \rceil)=1$ are automatically satisfied \cite[Remark 2.2]{Amb04}.

Assume $M'$ is numerically trivial along $X' \subs \overline{\eta}_{Z'}.$.
Then, by \cite[Theorem 1.2]{Gon13}, we have $\K X'_{\overline{\eta}_{Z'}}. + B'_{\overline{\eta}_{Z'}} \sim_\Q M'_{\overline{\eta}_{Z'}} \sim_\Q 0$.
Therefore, there is a dense open subset $U' \subset Z'$ such that $M'|\subs X'_{U'}. \sim_{\Q,U'} 0$, where $X' \subs U'.$ denotes the inverse image of $U'$ in $X'$.
Let $X \subs U'.$ be the inverse image of $U'$ in $X$.
Write $\alpha  \colon  X \rar X'$.
By the negativity lemma, $M = \alpha^* M' - E$, where $E \geq 0$ is $\alpha$-exceptional.
As $M' \subs U'.$ is trivial along the fibers of $X' \subs U'. \rar U'$ and $M$ is nef, $E$ does not dominate $Z'$.
Therefore, up to shrinking $U'$, we have that $M$ is trivial over $U'$.
Now, by Remark \ref{base change property}, we may assume that $X \rar Z'$ is weakly semi-stable.
We have $M \sim \subs \Q,Z'. F$, where $F$ is exceptional over $Z'$.
As the morphism $X \rar Z'$ is flat, the image in $Z'$ of each component of $F$ is a divisor.
Thus, up to replacing $F$ in its $\Q$-linear equivalence over $Z'$, we may assume that $F \geq 0$ and that $F$ is very exceptional over $Z'$ \cite[Definition 3.1]{Bir12}.
Then, by \cite[Lemma 3.3]{Bir12}, we have $F=0$.
Thus, it follows that $M \sim \subs \Q,Z'. 0$.
Therefore, we can apply the classic theory \cite{Amb04,FG14}.

Hence, we may assume that $M' \subs \overline{\eta}_{Z'}.$ is not numerically trivial.
In particular, $\K X'. + (B')^h$ is not pseudo-effective over $Z'$.
So, we can run a $(\K X'. + (B')^h)$-MMP relative to $Z'$ with scaling of an ample divisor.
Since $(X',(B')^h)$ is dlt and $\K X'. + (B')^h$ is not pseudo-effective over $Z'$, this terminates with a Mori fiber space $g  \colon  X'' \rar Y$ over $Z'$ \cite{BCHM}.
Then, we can apply Lemma \ref{lemma MFS} to $g  \colon  X'' \rar Y$ and $(X'',B''+M'')$, as $M''$ is ample over $Y$ by construction.
We induce a generalized pair $(Y,\bB Y. + \bM Y.)$ on $Y$.

In case $\dim Y = \dim Z'$, $Y \rar Z'$ is birational, and we are done.
In particular, this proves the case when $\dim X' - \dim Z' = 1$.
Thus, we may assume that $\dim Z' < \dim Y$.
Notice that $Y \rar Z'$ has connected fibers.
Then, by construction and Proposition \ref{inv adj fiber spaces}, the generalized pair $(Y,\bB Y. + \bM Y.)$ and the fibration $Y \rar Z'$ satisfy the hypotheses of the statement with smaller relative dimension.
By the inductive hypothesis, the statement applies and provides a $\Q$-Cartier b-nef b-divisor on $Z'$.
By Lemma \ref{lemma diagram of fibrations}, it is the same b-divisor induced by $(X,B+M)$.
This proves the inductive step.
\end{proof}

As an immediate consequence, we recover a result due to Chen and Zhang \cite{CZ13}.
The idea to apply Theorem \ref{full generality theorem firbations} to the following setup was suggested by Jingjun Han.

\begin{corollary}[{\cite[Main Theorem]{CZ13}}] \label{corollary jingjun}
Let $(X,B)$ be a projective log canonical pair such that $-(K_X+B)$ is nef.
Let $f \colon  X \rar Y$ be a surjective morphism, where $Y$ is projective and $\K Y.$ is $\Q$-Cartier.
Then $-\K Y.$ is pseudo-effective.
\end{corollary}

\begin{proof}
Define $M \coloneqq -(\K X. + B)$.
Then, the generalized pair $(X,B+M)$ is generalized log canonical.
Let $g \colon  X \rar Z$ and $h  \colon  Z \rar Y$ be the morphisms induced by the Stein factorization of $f$.

Then, by Theorem \ref{full generality theorem firbations}, $(X,B+M)$ induces a generalized log canonical pair $(Z,\bB Z. + \bM Z.)$ on $Z$.
Set $B_Z \coloneqq \bB Z,Z.$ and $M_Z \coloneqq \bM Z,Z.$.
Since $\K X. + B + M \sim_\Q 0$, we have $\K Z. + B_Z + M_Z \sim_\Q 0$. 

Now, since $h$ is finite, by the Riemann-Hurwitz formula, we have $h^* \K Y.= \K Z. - R$, where $R \geq 0$.
Thus, we get
$$
- \K Z. + R \sim_\Q B_Z + M_Z + R.
$$
As $B_Z \geq 0$ and $M_Z$ is pseudo-effective, $-\K Z. + R$ is pseudo-effective.
Thus, as $h$ is finite, we can apply \cite[Theorem 1.20]{Fuj83} to $-\K Y. + tA$ for $A$ ample on $Y$ and $0< t \ll 1$ to conclude that $- \K Y.$ is pseudo-effective.
\end{proof}

Now, using ideas of Fujino and Gongyo \cite{FG14}, we study the relation between the generalized pair induced on $Z'$ by $(X',B'+M')$ and the one induced by a generalized log canonical center of $(X',B'+M')$ dominating $Z'$.

\begin{theorem} \label{FG comparison}
Let $(X',B'+M')$ be a {projective} generalized sub-pair with data $X \rar X'$ and $M$.
Assume that $B'$, $M'$ and $M$ are $\Q$-divisors.
Let $f  \colon  X' \rar Z'$ be a contraction such that $\K X'. + B' + M' \sim \subs \Q, f. 0$.
Also, let $(X',B'+M')$ be generalized log canonical and generalized dlt over the generic point of $Z'$.
Then, for any generalized log canonical center $W'$ of $(X',B'+M')$ dominating $Z'$, we have $\bM Z'.= \bM Z'. \sups W'.$ and $\bB Z'.= \bB Z'. \sups W'.$, where these are the b-divisors induced on $Z'$ by $X'$ and $W'$ respectively.
\end{theorem}

\begin{remark} \label{remark compare}
{\em
In the setup of Theorem \ref{FG comparison}, $W'$ is a stratum of $(B')^h$, and therefore inherits a structure of generalized pair $(W',B \subs W'. + M \subs W'.)$ by repeated divisorial adjunction.
By \cite[Definition 4.7]{BZ}, $(W',B \subs W'. + M \subs W'.)$ is generalized log canonical over the generic point of $Z'$.
Also, notice that $W' \rar Z'$ may not have connected fibers.
Therefore, by comparing the b-divisors induced by $X'$ and $W'$, we implicitly allow generically finite base changes, which are allowed by Remark \ref{base change property}.
}
\end{remark}

\begin{proof}
First, we reduce to the case when $W'$ is a divisor.
Fix a generalized log canonical center $W'$ as in the statement and let $W'=P'_1 \cap \ldots \cap P'_k$, where each $P_i'$ is a prime component of $(B')^h$ of coefficient $1$.
Further assume that the statement is true if $\dim X' - \dim W' =1$.
In particular, we have that $(X',B'+M')$ and $(P_1',B \subs P_1'. + M \subs P_1'.)$ induce the same b-divisors.
Let $P'_{1,k}$ the restriction of $P_k'$ to $P_1'$.
Up to pulling back divisors from $Z$ to guarantee $(B')^v \geq 0$, we can take a weak generalized dlt model $P_1$ of $(P_1',B \subs P_1'. + M \subs P_1'.)$.
Up to considering the Stein factorization of $P_1 \rar Z'$, $(P_1,B \subs P_1. + M \subs P_1.)$ satisfies the hypotheses of the statement.
Here $(P_1,B \subs P_1. + M \subs P_1.)$ denotes the trace of $(P_1',B \subs P_1'. + M \subs P_1'.)$ on $P_1$.
Let $P_{1,2}$ be the strict transform of $P'_{1,2}$ to $P_1$.
Now, since we are assuming the statement in case of generalized log canonical centers of codimension 1, we have that $P_1$ and $P_{1,2}$ induce the same b-divisors on the base.
On the other hand, the generalized pair structure induced on $P_{1,2}$ by $(P_1,B \subs P_1. + M \subs P_1.)$ agrees with the one induced by $(X',B'+M')$ on $P'_{1,2}$.
Therefore, $X'$ and $P'_{1,2}$ induce the same b-divisors on the base.
Repeating this argument $k-1$ times, we get the claimed reduction.

From now on, we may assume that $W'$ is a prime divisor such that $\mult \subs W'. (B')^h = 1$.
Up to pulling back some effective divisors on $Z'$ to guarantee that $(B')^v \geq 0$, we can apply Theorem \ref{generalized weak dlt model}.
Thus, we may assume that $X'$ is $\Q$-factorial and that $(X',(B')^h)$ is dlt.
Let $W' \rar Y'$ be the Stein factorization of $W' \rar Z'$.

By a generically finite base change $T \rar Z'$ factoring through $Y'$, we may assume that the following properties hold \cite[cf. proof of Theorem 1.1]{FG14}:
\begin{itemize}
\item $V'$, the normalization of the main component of $X' \times \subs Z'. T$, has a semi-stable resolution in codimension 1 \cite[see Theorem 4.3]{Amb04}.
Call the latter $V$. Let $(V',B \subs V'. + M \subs V'.)$ be the generalized sub-pair induced by $(X',B'+M')$.
We may assume that $V$ is also a semi-stable resolution of a suitable higher model of $(X,B)$, a higher model of $X'$ where $M$ descends.
Therefore, we may assume that the moduli b-divisor $\bM V'.$ descends to $V$.
We will write $M_V \coloneqq \bM V',V.$; and
\item there are a fibration $U \rar T$ and a generalized sub-pair $(U,\bB U. + \bM U.)$ induced by $W' \rar Z$ and $(W',\bB W'. + \bM W'.)$.
By construction, $U$ maps birationally onto a prime divisor $\Gamma' \subset V'$ such that $\mult \subs \Gamma'. B \subs V'. = 1$.
Notice that, by construction, the generalized sub-pair structure induced by $(V',\bB V'. + \bM V'.)$ on $\Gamma'$ agrees with $(U,\bB U. + \bM U.)$.
\end{itemize}

Now, let $B_T$ and $B_T \sups \mathrm{min}.$ the boundary divisors induced by $V' \rar T$ and $U \rar T$ respectively.
Analogously, we have moduli divisors $M_T$ and $M_T \sups \mathrm{min}.$.
By construction \cite[cf. proof of Theorem 1.1]{FG14}, we have
$$
\K T. + B_T + M_T \sim_\Q \K T. + B_T \sups \mathrm{min}. + M_T \sups \mathrm{min}..
$$
Then, to conclude, it suffices to show $B_T = B_T \sups \mathrm{min}.$.

Taking hyperplane sections, we may assume that $T$ is a curve.
Therefore, $\pi \colon  (V,B_V) \rar T$ is semi-stable.
Thus, $\Supp (B_V) \cup \pi^* Q$ is a reduced simple normal crossing divisor for every $Q \in T$.
In particular, there is a finite set $\Sigma \subset T$ such that
$$
B_T = \sum \subs P \in \Sigma. (1-b_P)P, \quad B_T \sups \mathrm{min}.= \sum \subs P \in \Sigma. (1-b_P \sups \mathrm{min}.)P
$$
and all the singular fibers of $\pi \colon V \rar T$ are mapped to $\Sigma$.

Let $\mathcal{E}=\lbrace E_i \rbrace \subs i=1.^l$ the set of prime divisors on $V$ such that $\pi (E_i) \in \Sigma$ and 
$$
\mult \subs E_i. (B_V + \sum \subs P \in \Sigma. b_P \pi^* P)\sups \geq 0 . <1.
$$
Fix a rational number $0 < \epsilon \ll 1$. Then, $(V,(B_V + \sum \subs P \in \Sigma. b_P \pi^* P)\sups \geq 0. + \epsilon \sum E_i)$ is dlt.
Notice that we have
$$
\K V. + (B_V + \sum \subs P \in \Sigma. b_P \pi^* P)\sups \geq 0. + \epsilon \sum E_i + M_V \sim \subs \Q,T. - (B_V + \sum \subs P \in \Sigma. b_P \pi^* P)\sups \leq 0. + \epsilon \sum E_i \eqqcolon E.
$$

By \cite{BZ}, we can run a $(\K V. + (B_V + \sum \subs P \in \Sigma. b_P \pi^* P)\sups \geq 0. + \epsilon \sum E_i + M_V)$-MMP with scaling of an ample divisor. Notice that every component of $E$ that dominates $T$ is exceptional over $V'$, and is not contained in the relative movable cone over $V'$ \cite[Definition 2.1]{F11b}.
So, if there are components of $E$ dominating $T$, we first run an MMP on $\K V. + (B_V + \sum \subs P \in \Sigma. b_P \pi^* P)\sups \geq 0. + \epsilon \sum E_i + M_V$ relative to $V'$ with scaling of an ample divisor $H$.

Assume by contradiction that this MMP does not terminate.
Let $\lambda \geq 0$ be the limit of the coefficients used in the scaling by $H$.
If $\lambda > 0$, then the MMP is an MMP for $\K V. + (B_V + \sum \subs P \in \Sigma. b_P \pi^* P)\sups \geq 0. + \epsilon \sum E_i + M_V + \lambda H$.
Since $( V, (B_V + \sum \subs P \in \Sigma. b_P \pi^* P)\sups \geq 0. + \epsilon \sum E_i)$ is dlt and $\lambda H + M_V$ is ample, we get a contradiction by \cite[Theorem 2.3]{F11b}.
If $\lambda = 0$, we have that $\K V_j . + (B_V + \sum \subs P \in \Sigma. b_P \pi^* P)_j \sups \geq 0. + \epsilon \sum E_i)_j + M \subs V_j. + \lambda_j H_j + G_j$ is ample, where $V_j$ is the $j$-th model in the MMP and $G_j$ is an ample divisor on $V_j$.
We can choose the divisors $G_j$ such that the sequence of strict transforms on $V$ converges to 0 in $N^1(V/V')$.
Since $\lim \lambda_j = \lambda = 0$, we get that $\K V. + (B_V + \sum \subs P \in \Sigma. b_P \pi^* P)\sups \geq 0. + \epsilon \sum E_i + M_V$ is limit of divisors movable over $V'$, and hence movable.
This provides a contradiction.
By repeatedly applying \cite[Lemma 2.9]{Lai11} at each step of the MMP, it follows that all the components of $E$ dominating $T$ are contracted.

Thus, after running an MMP over $V'$, we get to a model $V''$.
Since the MMP just run is an MMP for $\K V. + (B_V + \sum \subs P \in \Sigma. b_P \pi^* P)\sups \geq 0. + \epsilon \sum E_i + M_V + \sigma H$ for some $\sigma > 0$, we can turn $M + \sigma H$ into a boundary and conclude that $V''$ is a $\Q$-factorial klt variety and $(V'', ((B_V + \sum \subs P \in \Sigma. b_P \pi^* P)\sups \geq 0.)'' + (\epsilon \sum E_i)'')$ is dlt.
Thus, we can run an MMP for $\K V''. + ((B_V + \sum \subs P \in \Sigma. b_P \pi^* P)\sups \geq 0.)'' + (\epsilon \sum E_i)'' + M_{V''}$ relative to $T$ with scaling of an ample divisor.
Notice that $E''$, the image of $E$ on $V''$, is vertical over $T$ and $\Supp (E'')$ contains no fiber.
Therefore, $E''$ is not in the relative movable cone over $T$. By a similar argument as before, the MMP terminates, and contracts all of $E''$.
Call $\hat{V}$ the final model.
By construction, $\hat{E}=0$.
In particular, this guarantees that $(B \subs \hat{V}. + \sum \subs P \in \Sigma. b_P \hat{\pi}^* P)\sups \geq 0. = B \subs \hat{V}. + \sum \subs P \in \Sigma. b_P \hat{\pi}^* P$.
Also, by similar observations as before, we have that $(\hat{V},B \subs \hat{V}. + \sum \subs P \in \Sigma. b_P \hat{\pi}^* P)$ is dlt and $\hat{V}$ is $\Q$-factorial.
Furthermore, $(\hat{V},B \subs \hat{V}. + M \subs \hat{V}.)$ is the trace of the generalized pair $(\hat{V},\bB \hat{V}. + \bM \hat{V}.)$.
Finally, notice that just the components of $E$ are contracted in this step.
In particular, the strict transform of $\Gamma'$ is not contracted and it is normal, as $(\hat{V},B \subs \hat{V}. + \sum \subs P \in \Sigma. b_P \hat{\pi}^* P)$ is dlt.

Let $\hat{\Gamma}$ be the strict transform of $\Gamma$ on $\hat{V}$.
By construction, the generalized pair $(\hat{\Gamma},B \subs \hat{\Gamma}. + M \subs \hat{\Gamma}.)$ induced on it by $(\hat{V},B \subs \hat{V}. + M \subs \hat{V}.)$ is crepant to $(U, B_U + M_U)$.
Thus, the generalized pair $(\hat{\Gamma},\Delta \subs \hat{\Gamma}. + M \subs \hat{\Gamma}.)$ induced on it by $(\hat{V},B \subs \hat{V}. +  \sum \subs P \in \Sigma. b_P \hat{\pi}^* P + M \subs \hat{V}.)$ is crepant to $(U, B_U +  \sum \subs P \in \Sigma. b_P \pi_U^* P + M_U)$.
Here $\hat{\pi}$ and $\pi_U$ denote the morphisms to $T$ from $\hat{V}$ and $U$ respectively.

By construction, we have
$$
B_{\hat{V}} + \sum \subs P \in \Sigma. b_P \hat{\pi}^*P \geq \sum \subs P \in \Sigma. \hat{\pi}^*P.
$$
Thus, $(\hat{V},B \subs \hat{V}. +  \sum \subs P \in \Sigma. b_P \hat{\pi}^* P + M \subs \hat{V}.)$ is generalized log canonical, while the generalized pair $(\hat{V},B \subs \hat{V}. + \sum \subs P \in \Sigma. (b_P+\delta) \hat{\pi}^* P + M \subs \hat{V}.)$ is not generalized log canonical along the divisor $\sum \subs P \in \Sigma. \hat{\pi}^*P$ for any $\delta > 0$.
Since $\hat{V}$ is $\Q$-factorial and $(\hat{V},\hat{\Gamma})$ is plt, we can apply \cite[Lemma 3.2]{Bir16a}.
In particular, the generalized pair induced by $(\hat{V},B \subs \hat{V}. + \sum \subs P \in \Sigma. (b_P+\delta) \hat{\pi}^* P + M \subs \hat{V}.)$ on $\hat{\Gamma}$ is not generalized log canonical along $\sum \subs P \in \Sigma. \hat{\pi}^*P \cap \hat{\Gamma}$. Since $(\hat{\Gamma},\Delta \subs \hat{\Gamma}. + M \subs \hat{\Gamma}.)$ is crepant to $(U, B_U +  \sum \subs P \in \Sigma. b_P \pi_U^* P + M_U)$, we have $b_P = b_P \sups \mathrm{min}.$, which completes the proof.
\end{proof}

\section{Generalized adjunction}

In this section, we use the machinery developed for fiber spaces to define adjunction for higher codimensional generalized log canonical centers.
{The results of this sections hold for arbitrary generalized pairs $(X',B'+M')$ with data $X \rar X' \rar V$ and $M$, without the assumption $V= \Spec(\C)$.}

As in the classic case, given a generalized log canonical center $W' \subset X'$, the idea is to extract a generalized log canonical place $E$ dominating $W'$.
Then, one can apply generalized divisorial adjunction on $E$ and apply the generalized canonical bundle formula to the morphism $E \rar W'$.
This leads to the following definition of \emph{generalized adjunction} in arbitrary codimension.

\begin{defn} \label{def gen adj} {\em
Let $(X',B'+M')$ be a generalized pair with data $X \rar X' {\rar V}$ and $M$.
Let $W'$ be a generalized log canonical center.
Fix a corresponding generalized log canonical place $E$.
We may assume that $E$ is a smooth divisor on $X$.
Then, we have an induced morphism $E \rar W^\nu$, where $W^\nu$ denotes the normalization of $W'$.
We define b-divisors $\bB W^\nu.$ and $\bM W^\nu.$ as the boundary and moduli part of fiber space adjunction for $(E,B_E+M_E)$ over $W^\nu$.}
\end{defn}

\begin{remark}{\em
In case $W'$ is not an exceptional generalized log canonical center, one has to prove that the definition does not depend on the choice of $E$.
{Furthermore, one needs to check that the induced morphism $E \rar W^\nu$ generically has connected fibers.}
We will address this in Remark \ref{first remark about well defined} and Theorem \ref{partial well posed}, giving a positive answer.}
\end{remark}

\begin{remark} \label{first remark about well defined} {\em
In case $M'$ descends in a neighborhood of $W'$, it does not contribute to the singularities along $W'$, and generalized adjunction coincides with the usual one.
Therefore, Definition \ref{def gen adj} is well posed if $W'$ is any generalized log canonical center with $M$ relatively trivial over $W'$.}
\end{remark}

\begin{remark} {\em
In case $W'$ is a divisor, this definition coincides with the divisorial generalized adjunction introduced by Birkar and Zhang \cite[Definition 4.7]{BZ}.}
\end{remark}

\begin{remark}{\em
If $M'$, $M$ and $B'$ are $\Q$-divisors, $\bB W^\nu.$ is defined via log canonical thresholds of $\Q$-divisors.
These are rational numbers.
Then, in this case, $\bB W^\nu.$ and $\bM W^\nu.$ are automatically $\Q$-Weil b-divisors.}
\end{remark}

Now, we would like to study the properties of the b-divisors $\bB W^\nu.$ and $\bM W^\nu.$.
In particular, we would like to show that a structure of a generalized pair is induced on $W^\nu$.
Going in this direction, we are ready to prove Theorem \ref{gen adj higher codim}.

\begin{proof}[Proof of Theorem \ref{gen adj higher codim}]
As generalized adjunction is known in case $W'$ is a divisor, we may assume $\codim W' \geq 2$.
Let $E$ be the generalized log canonical place corresponding to $W'$.
Without loss of generality, we may assume that $f \colon X \rar X'$ is a log resolution, and $E$ is a divisor on $X$.
Generalized divisorial adjunction provides us with a generalized sub-pair $(E,B_E+M_E)$, where the moduli part descends to $E$.
Then, as $W'$ is an exceptional center, $(E,B_E)$ is sub-klt over the generic point of $W'$.
Furthermore, by assumption, $(X',B'+M')$ is generalized log canonical.
Therefore, we can apply \cite[Lemma 4.5]{BZ} to extract just $E$ over $\eta \subs W'.$.
Call this model $X''$, and denote the image of $E$ by $E''$.
Then, $B_{E''}$ is effective.

Now, we check that the technical conditions needed to apply the canonical bundle formula are satisfied.
By assumption, up to shrinking $X'$ to a neighborhood of $W'$, we have $B = E + \Delta$, where $\lfloor \Delta \rfloor \leq 0$.
Then, we can write $\Delta$ in a unique way as $\Delta = \Gamma - A$, where $\Gamma \geq 0$, $\lfloor \Gamma \rfloor = 0$, and $A \geq 0$ is integral and $f$-exceptional.
Consider the short exact sequence
$$
0 \rar \O X. (A-E) \rar \O X.(A) \rar \O E.(A|_E) \rar 0.
$$
Then, the corresponding long exact sequence of higher direct images provides us with
$$
\O X'. \cong f_* \O X. (A) \rar f_* \O E. (A |_E) \rar R^1 f_* \O X. (A-E) =0,
$$
where the latter element vanishes by \cite[Corollary 2.68]{KM} applied to
$$
A-E \sim_\Q \K X. + \Gamma + M - f^*(\K X'. + B' + M').
$$
This forces the chain of equalities
\begin{equation} \label{equation for effective coeffs}
(f_E)_* \O E.(A|_E) = \O W^\nu. = \O W'.,
\end{equation}
where $f_E$ denotes the restriction of $f$ to $E$, and $\O W^\nu.$, the structure sheaf of the normalization of $W'$, is seen as a sheaf of $\O W'.$-modules.
Thus, $W'$ is normal.

Now, we can apply Theorem \ref{full generality theorem firbations}, which guarantees that $(W',\bB W'. + \bM W'.)$ is a generalized sub-pair.
Furthermore, by Proposition \ref{inv adj fiber spaces}, it is generalized sub-klt.

We are left with showing that $\bB W',W'.$ is effective.
As argued in \cite[proof of Theorem 8.6.1]{K07}, this follows by equation \eqref{equation for effective coeffs} and \cite[Theorem 8.3.7]{K07}.
\end{proof}

Now, we would like to address the case when $W'$ is not an exceptional generalized log canonical center.
{To do so, we need the following result about \emph{tie breaking} for generalized pairs.
The statement is a slight generalization of \cite[Proposition 8.7.1]{K07}.
}

\begin{proposition} \label{tie breaking}
{Let $(X',B'+M')$ be a generalized klt pair with data $X \rar X' \rar V$ and $M$.
Assume that $B'$, $M'$ and $M$ are $\Q$-divisors.
Let $D'$ be an effective $\Q$-Cartier divisor on $X'$ and $N$ a $\Q$-Cartier divisor on $X$ that is nef over $V$.
Further, assume that $D'+N'$ is $\Q$-Cartier, where $N'$ denotes the pushforward of $N$ to $X'$.
Assume that the generalized pair $(X',B'+D'+M'+N')$ with data $X' \rar X \rar V$ and $M+N$ is generalized log canonical.
Let $W' \subset X'$ be a minimal generalized log canonical center of $(X',B'+D'+M'+N')$.
Then, there exist an effective $\Q$-Cartier divisor $P'$ and a rational number $0 < \epsilon \ll 1$ such that $W'$ is an exceptional generalized log canonical center of $(X',B'+P'+M'+(1-\epsilon)(D'+N'))$.}
\end{proposition}

\begin{proof}
{The proof of \cite[Proposition 8.7.1]{K07} goes through almost verbatim.
Therefore, we just point the relevant changes in the proof.
Without loss of generality, we may assume that $f \colon X \rar X'$ is a log resolution of $(X',B'+D')$.
Then, we may write
\begin{equation} \label{eq tie break}
f^*(\K X'. + B' + M' = \K X. + \sum b_i E_i + M, \quad \text{and} \quad f^*(D'+N')= \sum a_i E_i + N.
\end{equation}
After redefining $a_i$ and $b_i$ in the proof of \cite[Proposition 8.7.1]{K07} with the ones in equation \eqref{eq tie break}, the proof of \cite[Proposition 8.7.1]{K07} goes through.
Notice that the roles of $X$ and $X'$ are flipped in the notation of \cite[Proposition 8.7.1]{K07}.}
\end{proof}

\begin{theorem} \label{partial well posed}
Let $(X',B'+M')$ be a {projective} generalized pair with data $X \rar X' {\rar V}$ and $M$.
{Assume that $X'$ is a $\Q$-factorial klt variety.}
Let $B'$, $M'$ and $M$ be $\Q$-divisors.
Let $W'$ be a generalized log canonical center, and $W^\nu$ its normalization.
{Assume that $W'$ is projective.}
Then, Definition \ref{def gen adj} is well posed, and the induced moduli b-divisor $\bM W^\nu.$ is b-Cartier and b-nef.
Furthermore, $\bB W^\nu,W^\nu.$ is effective.
\end{theorem}

\begin{proof}
By assumption, $(X',B'+M')$ is generalized log canonical in a neighborhood $U'$ of $\eta \subs W'.$.
{Furthermore, for every $0 < \epsilon \ll 1$, the generalized pair $(X',(1-\epsilon)(B'+M'))$ with data $X \rar X' \rar V$ and $(1 - \epsilon )M$ is generalized klt in an open set containing the generic point of $W'$.
Up to shrinking $U'$, we may assume that $W' \cap U'$ is a minimal generalized log canonical center, and that the restriction of $(X',B'+M')$ to $U'$ is generalized log canonical.
Then, on the open set $U'$, we can apply Proposition \ref{tie breaking}.
In particular, there exist $0 < \delta \ll 1$ and an effective divisor $P'$ such that $W' \cap U'$ is an exceptional generalized log canonical center for the restriction of $(X',(1-\delta)(B'+M')+P')$ to $U'$.
Let $E$ be the corresponding generalized log canonical place.
Notice that $E$ is also a generalized log canonical place for $(X',B'+M')$.}

{Let $X'' \rar X'$ be a weak generalized dlt of $(X',(1-\delta)(B'+M')+P')$, and let $E''$ denote the trace of $E$ on $X''$.
Notice that, by construction, $E''$ is normal.
Since $W'$ is exceptional for $(X',(1-\delta)(B'+M')+P')$ along $U'$, we can apply the proof of Theorem \ref{gen adj higher codim}.
In particular, it follows that $W' \cap U'$ is normal and that the induced morphism $E'' \rar W'$ is a contraction over $W' \cap U'$.
Let $W''$ denote the Stein factorization of $E'' \rar W'$.
Then, we have $W'' = W^\nu$.}

{Now, let $X^m \rar X$ be a weak generalized dlt model of $(X',B'+M')$.
By Proposition \ref{weak generalized dlt model with extraction}, we may assume that $E$ appears as a divisor on $X^m$.
Let $E^m$ denote the trace of $E$ on $X^m$, and let $(E^m, B \subs E^m. + M \subs M^m.)$ denote the generalized pair induced by generalized adjunction.
Notice that $E^m$ and $E''$ are normal and birational to each other.
Furthermore, they both admit a morphism to $W^\nu$.
Since $E'' \rar W^\nu$ is a contraction, by taking a common resolution of $E^m$ and $E''$, we obtain that $E^m \rar W^\nu$ is a contraction.
}

{Let $(X^m,B^m+M^m)$ denote the pair obtained by taking the weak generalized dlt model, and let $(X^m,B^m+\Delta^m + M^m)$ denote the trace of $(X',B'+M')$ on $X^m$.
Notice that $\Delta^ \geq 0$.
Let $(E^m,B \subs E^m. + M \subs E^m.$ denote the generalized pair induced by $(X^m,B^m+\Delta^m+M^m)$ by generalized adjunction.
Since $E^m$ appears with coefficient 1 in $B^m$, $(X^m,B^m)$ is dlt, by \cite[Remark 4.8]{BZ} it follows that $B \subs E^m. \geq 0$.
Then, we can apply Theorem \ref{full generality theorem firbations} and define a structure $(W^\nu, \bB W^\nu. + \bM W^\nu.)$ of generalized sub-pair on $W^\nu$.
Since $B \subs E^m. \geq 0$, it follows that $B \subs W^\nu, W^\nu. \geq 0$.}

{Now, we are left with showing that the generalized pair structure $(W^\nu,\bB W^\nu. + \bM W^nu.)$ is intrinsic.
Let $E^m_1, \ldots , E^m_l$, where $l \geq 2$, be the distinct generalized log canonical places with center $W'$ appearing in $X^m$, with $E^m=E^m_1$.
By the choice of $X^m$, $(X^m,B^m)$ is dlt over $U'$.
By the connectedness principle \cite[Corollary 5.49]{KM}, the locus $E^m_1 \cup \ldots \cup E^m_l$ is connected over the generic point of $W'$.
Let $W^m_i$ be the Stein factorization of $E^m_i \rar W'$.
While $W^m_1=W^\nu$, it may be that $W^m_i \rar W^\nu$ is a finite morphism for $i \geq 2$.
Thus, each $E^m_i$ induces a generalized pair structure on a finite cover of $W^\nu$.
In the spirit of Remark \ref{remark compare}, we can compare these generalized pair structures after a generically finite base change.
By abuse of language, we omit the base change and talk about generalized pair structures induced on $W^\nu$.
}

{Fix $E^m_k$ for some $1 \leq k \leq l$, and denote by $(E^m_k,B \subs E^m_k. + M \subs E^m_k.)$ the generalized pair induced by divisorial adjunction.
Then, $(E_k,B \subs E_k. + M \subs E_k.)$ is generalized log canonical and generalized dlt over the generic point of $W'$.
Now, by Theorem \ref{FG comparison}, the generalized pair induced by $(E^m_k,B \subs E^m_k. + M \subs E^m_k.)$ on $W^\nu$ is the same as the one induced by a generalized log canonical center $F$ of $(E_k,B \subs E_k. + M \subs E_k.)$ that dominates $W^\nu$.
By the construction of $X^\nu$, every such $F$ arises as the intersection of some of the $E^m_i$'s.
}

Thus, if $E^m_k \cap E^m_j \neq \emptyset$, $(E^m_k,B \subs E^m_k. + M \subs E^m_k.)$ and $(E^m_j,B \subs E^m_j. + M \subs E^m_j.)$ induce the same b-divisors on $W^\nu$.
As $E^m_1 \cup \ldots \cup E^m_l$ is connected over the generic point of $W'$, by transitivity, we conclude that all the $E^m_i$'s induce the same b-divisors.
\end{proof}

We conclude discussing inversion of adjunction in the setup of generalized pairs. The strategy follows the lines of \cite{Hac14}.

\begin{proof}[Proof of Theorem \ref{prop inv of adj}]
The ``only if'' direction follows immediately by divisorial adjunction applied to a generalized log canonical place, and Proposition \ref{inv adj fiber spaces}.
Thus, we are left with proving the ``if'' part.
Proceeding by contradiction, henceforth we will assume that $(W^\nu,\bB W^\nu. + \bM W^\nu.)$ is generalized log canonical, while $(X',B'+M')$ is not generalized log canonical near $W'$.

Let $(X^m,B^m + M^m)$ be a $\Q$-factorial weak generalized dlt model for $(X',B'+M')$.
By Proposition \ref{weak generalized dlt model with extraction}, we may assume that there is a generalized log canonical place $S^m$ corresponding to $W'$.
Set $\Sigma^m \coloneqq (E^+ - E)^m$, and $\Gamma^m \coloneqq B^m - S^m$.
As in the proof of Theorem \ref{generalized weak dlt model}, $E^+$ denotes the divisors on $X$ of generalized discrepancy at most $-1$, and $E \coloneqq \red E^+$.

By the proofs of Theorem \ref{generalized dlt model} and Proposition \ref{weak generalized dlt model with extraction}, there is a big divisor $L^m$ with the following property: for any $t > 0$ we can find a divisor $\Theta^m_t \sim_\R \Gamma^m + tL^m + M^m$ such that $(X^m, S^m + \Theta_t^m)$ is plt.
Up to twisting $L^m$ by an ample divisor, we may assume that we can run the relative $(\K X^m. + B^m + M^m)$-MMP over $X'$ with scaling of $L^m$.

Let $\phi_i  \colon  X^m_i \drar X^m_{i+1}$ be the sequence of flips and divisorial contractions, and let $\mu_i  \colon  X^m_i \rar X'$ and $\bar{\mu}_i  \colon  S^m_i \rar S'$ be the induced morphisms.
Then, there is a sequence of non-negative rational numbers $\lbrace s_i \rbrace_{i \geq 0}$ such that $s_i \geq s_{i+1}$, and either $s_{N+1}=0$ for some $N \in \N$, or $\lim _{i \to + \infty}s_i=0$.
Furthermore, the divisor $\K X_i^m. + S^m_i + \Gamma^m_i + M^m_i + s L^m_i$ is nef over $X'$ for all $s_i \geq s \geq s_{i+1}$.
As the plt property is preserved by steps of the MMP \cite[Corollary 3.44]{KM}, the pair $(X^m_i,S^m_i + \Theta^m_{t,i})$ is plt if $t < s_i$.

By standard arguments, we may assume that $\phi_i$ is a flip for $i \geq i_0$, for some $i_0 \in \N$.
In addition, by the arguments in the proofs of Step 1 and Step 2 of \cite[\nopp 4.2.1]{Fuj07}, we may assume that $S^m_i \drar S^m_{i+1}$ is an isomorphism in codimension 1 for all $i \geq i_0$.

Now, assume that for some $i \geq 0$ we have $S^m_i \cap \Sigma^m_i \neq \emptyset$. Then, we can write
$$
\mu_i^*(\K X'. + B' +M')|_{S^m_i} = \K S_i^m. + B_{S^m_i} + M_{S^m_i},
$$
and $(B_{S^m_i})^{>1} \neq 0$.
In particular, $(S_i^m,B_{S^m_i} + M_{S^m_i})$ is not generalized log canonical. Consider the induced fibration $S^m_i \rar W^\nu$.
Then, by Proposition \ref{inv adj fiber spaces} and the construction of generalized adjunction, $(W^\nu,\bB W^\nu. + \bM W^\nu.)$ is not generalized log canonical.
This is a contradiction.
{Notice that, as in the proof of Theorem \ref{partial well posed}, we are abusing notation, as $S^m_i \rar W^\nu$ may not have connected fibers.
On the other hand, this is not a problem, since a generalized sub-pair is generalized sub-log canonical if and only if the generalized sub-pair induced by a generically finite base change is sub-log canonical \cite[Proposition 5.20]{KM}.
}

Thus, we may assume that $S_i^m \cap \Sigma_i^m = \emptyset$ for all $i \geq 0$.
For any integer $k \gg 0$ such that $k \Sigma^m$ is an integral divisor, pick $i \geq i_0$ such that $s_i > \frac{1}{k} \geq s_{i+1}$.
Then,
$$
L^m_i - k \Sigma^m_i -S^m_i \sim_{\Q,\mu_i} \K X_i^m. + \Theta_{\frac{1}{k},i} + (k-1)\left(\K X^m_i. + S^m_i + \Gamma_i^m + M^m_i + \frac{1}{k}L^m_i\right).
$$
The pair $(X_i^m,\Theta_{\frac{1}{k},i})$ is klt, while $\K X^m_i. + S^m_i + \Gamma_i^m + M^m_i + \frac{1}{k}L^m_i$ is $\mu_i$-nef.
Hence, by the relative version of Kawamata--Viehweg vanishing \cite[cf. proof of Corollary 2.68 and Theorem 2.70]{KM}, we have $R^1\mu_{i,*}\O X_i^m. (L^m_i - k \Sigma^m_i -S^m_i)=0$.
This implies that there is a surjection
\begin{equation} \label{eq surjection}
\mu_{i,*}\O X_i^m. (L^m_i - k \Sigma^m_i) \rar \bar{\mu}_{i,*}\O S_i^m. (L^m_i - k \Sigma^m_i) = \bar{\mu}_{i,*}\O S_i^m. (L^m_i) \rar 0,
\end{equation}
where the equality $\bar{\mu}_{i,*}\O S_i^m. (L^m_i - k \Sigma^m_i) = \bar{\mu}_{i,*}\O S_i^m. (L^m_i)$ follows from the fact that $S_i^m \cap \Sigma_i^m = \emptyset$.
On the other hand, for $k \gg 0$ the subsheaves
$$
\mu_{i_0,*}\O X_{i_0}^m. (L^m_{i_0} - k \Sigma^m_{i_0}) \subset \mu_{i_0,*}\O X_{i_0}^m. (L^m_{i_0})
$$
are contained in $\mathcal{I}_{\mu_{i_0}(\Sigma^m_{i_0})} \cdot \mu_{i_0,*}\O X_{i_0}^m. (L^m_{i_0})$.
Since we are assuming that $(X',B'+M')$ is not generalized log canonical around $W'$, we have $W' \cap \mu_{i_0}(\Sigma^m_{i_0}) \neq \emptyset$.
Thus, the restrictions to $W'$ of the local sections of $\mu_{i_0,*}\O X_{i_0}^m. (L^m_{i_0} - k \Sigma^m_{i_0})$ vanish along the intersection $W' \cap \mu_{i_0}(\Sigma^m_{i_0})$.
Consequently, the morphism
\begin{equation} \label{eq not surjection}
\mu_{i_0,*}\O X_{i_0}^m. (L^m_{i_0} - k \Sigma^m_{i_0}) \rar \bar{\mu}_{i_0,*}\O S_{i_0}^m. (L^m_{i_0})
\end{equation}
is not surjective.

Now, since for $i \geq i_0$ all the maps $X_{i_0}^m \drar X \subs i.^m$ and $S_{i_0}^m \drar S \subs i.^m$ are isomorphisms in codimension $1$, we have equalities $\mu_{i,*}\O X_i^m. (L^m_i - k \Sigma^m_i) = \mu_{i_0,*}\O X_{i_0}^m. (L^m_{i_0} - k \Sigma^m_{i_0})$, and $\bar{\mu}_{i,*}\O S_i^m. (L^m_i) = \bar{\mu}_{i_0,*}\O S_{i_0}^m. (L^m_{i_0})$.
Thus, equation \eqref{eq not surjection} is equivalent to saying that
$$
\mu_{i,*}\O X_i^m. (L^m_i - k \Sigma^m_i) \rar  \bar{\mu}_{i,*}\O S_i^m. (L^m_i)
$$
is not surjective.
This contradicts equation \eqref{eq surjection}, and concludes the proof.
\end{proof}

\section{Applications to a conjecture of Prokhorov and Shokurov}

Now, we would like to discuss a possible application of the canonical bundle formula for generalized pairs.
In particular, we are interested in the connections with a conjecture by Prokhorov and Shokurov \cite[Conjecture 7.13]{PS}.
We start by recalling its statement.

\begin{conjecture}[{\cite[Conjecture 7.13]{PS}}] \label{Prokhorov-Shokurov conjecture}
Let $(X,B)$ be a sub-pair, and assume that $B$ is a $\Q$-divisor.
Let $f \colon X \rar Z$ be a contraction such that $\K X. + B \sim_{\Q,f} 0$.
Also, let $(X,B)$ be klt over the generic point of $Z$.
Then, we have:
\begin{itemize}
\item[(i)] $\bM Z.$ is b-semi-ample;
\item[(ii)] let $X \subs \eta.$ be the generic fiber of $f$.
Then $I_0(\K X_\eta. + B_\eta) \sim 0$, where $I_0$ depends only on $\dim X_\eta$ and the multiplicities of $B^h$; and
\item[(iii)] $\bM Z.$ is effectively b-semi-ample.
There exists a positive integer $I_1$ depending only on the dimension of $X$ and the horizontal multiplicities of $B$ (a finite set of rational numbers) such that $I_1 \bM Z.$ is very b-semi-ample; that is, $I_1 \bM Z.= \overline{L}$, where $L$ is a basepoint-free divisor on some birational model of $Z$.
\end{itemize}
\end{conjecture}

In view of the recent developments, we propose the following generalization of Conjecture \ref{Prokhorov-Shokurov conjecture}.

\begin{conjecture} \label{generalized conjecture}
Let $(X',B'+M')$ be a generalized sub-pair with data $X \rar X'$ and $M$.
Assume that $B'$, $M'$ and $M$ are $\Q$-divisors.
Let $f \colon X' \rar Z'$ be a contraction such that $\K X'. + B' + M' \sim_{\Q,f} 0$.
Assume that $M$ is semi-ample, and let $c$ be the minimum positive integer such that $|cM|$ is basepoint-free.
Also, let $(X',B'+M')$ be generalized klt over the generic point of $Z'$.
Then, we have:
\begin{itemize}
\item[(i)] $\bM Z'.$ is b-semi-ample;
\item[(ii)] let $X' \subs \eta.$ be the generic fiber of $f$.
Then $I_0(\K X'_\eta. + B'_\eta + M'_\eta) \sim 0$, where $I_0$ depends only on $\dim X'_\eta$, the multiplicities of $B^h$ and $c$, and we are free to replace the representative of $M'$ in its $\Q$-linear equivalence class; and
\item[(iii)] $\bM Z.$ is effectively b-semi-ample.
There exists a positive integer $I_1$ depending only on the dimension of $X'$, the horizontal multiplicities of $B$ and $c$ (a finite set of rational numbers) such that $I_1 \bM Z'.$ is very b-semi-ample; that is, $I_1 \bM Z'.= \overline{L}$, where $L$ is a basepoint-free divisor on some birational model of $Z'$.
\end{itemize}
\end{conjecture}

In the hope of a possible inductive approach, it is important to relate the two conjectures.

\begin{theorem} \label{PS conjecture implies generalized conj}
If Conjecture \ref{Prokhorov-Shokurov conjecture} is true in relative dimension $n$, then so is Conjecture \ref{generalized conjecture}.
More precisely, each part of Conjecture \ref{Prokhorov-Shokurov conjecture} implies the corresponding part of Conjecture \ref{generalized conjecture}.
\end{theorem}

\begin{proof}
Throughout the proof, we fix the relative dimension of the fibrations.
Also, let $(X',B'+M')$ be as in Conjecture \ref{generalized conjecture}.

Now, assume that part (i) of Conjecture \ref{Prokhorov-Shokurov conjecture} holds.
Let $Z$ be a higher model of $Z'$ where $\bM Z'.$ descends.
We may also assume that $|\bM Z',Z.|_\Q$ is resolved, that is $|\bM Z',Z.|_\Q = |L|_\Q + F$, where $|L|_\Q$ is a basepoint-free $\Q$-linear series and $F \geq 0$.
{Notice that, by Remark \ref{remark rel dim 1 M almost nef and good}, we know that $|\bM Z',Z.|_\Q \neq \emptyset$.}
Arguing by contradiction, we have $F \neq 0$.
Let $P$ be a prime divisor contained in the support of $F$.
By Remark \ref{remark min max}, there is $0 \leq A \sim_\Q M$ such that $\mult_P \bM Z'. = \mult_P \bM Z'.^A$.
Up to replacing $Z$ with a higher model, we may assume that $\bM Z'.^A$ descends to $Z$.
By assumption, $|\bM Z',Z.^A|_\Q$ is basepoint-free.
By construction, $\bM Z',Z. = \bM Z',Z.^A + E$, where $E \geq 0$ and $\mult_P E = 0$.
Thus, $|\bM Z',Z.|_\Q$ is free at the generic point of $P$, which is the required contradiction.

Now, assume that part (ii) of Conjecture \ref{Prokhorov-Shokurov conjecture} holds.
By assumption, we may write $\K X'.+B'+M' \sim_\Q \K X'. + B' + A'$, where $A'$ is horizontal over $Z'$, irreducible, and with coefficient $\frac{1}{c}$.
Thus, part (ii) of Conjecture \ref{generalized conjecture} follows.

Now, we are left with showing that part (iii) of Conjecture \ref{generalized conjecture} holds if we assume part (iii) of Conjecture \ref{Prokhorov-Shokurov conjecture}.
Notice that the arguments in \cite[Section 3]{Flo14} go through verbatim in the setup of generalized pairs.
Therefore, we can reduce to the case when the base $Z'$ is a smooth curve.
Let $I$ be the integer guaranteed by part (iii) of Conjecture \ref{Prokhorov-Shokurov conjecture}, where we allow $\coeff ( B'_\eta ) \cup \lbrace \frac{1}{c} \rbrace$ as the set of coefficients.
Then, by Remark \ref{remark min max} $I \bM Z',Z'.$ is integral.
Thus, it suffices to show it is basepoint-free.

By Remark \ref{remark min max}, for every point $P \in Z'$, there is $A \sim_c M$ such that $I \bM Z',Z'. = I \bM Z',Z'.^A + \sum \subs i=1.^k n_i P_i$, where $P \neq P_i$ for all $i$ and $n_i \in \N$.
Since $I \bM Z',Z'.^A$ is basepoint-free, $I \bM Z',Z'.$ is free at $P$. As $P$ is arbitrary, the claim follows.
\end{proof}

As an immediate corollary, we have the following.

\begin{corollary} \label{generalized ps conj curves}
Conjecture \ref{generalized conjecture} holds true if the relative dimension is 1.
\end{corollary}

\begin{proof}
It follows immediately from Theorem \ref{PS conjecture implies generalized conj} and \cite[Theorem 8.1]{PS}.
\end{proof}

Now, we are ready to show how to use Corollary \ref{generalized ps conj curves} in order to prove certain cases of Conjecture \ref{Prokhorov-Shokurov conjecture}.

\begin{proof}[Proof of Theorem \ref{prokhorov shokurov rel dim 2}]
{Up to replacing $Z$ with a higher model and $X$ with the normalization of the fiber product, by Chow's lemma, we may assume that $Z$ is quasi-projective.
Then, by \cite{HX13}, we may assume that $Z$ is projective.
}
Let $X \subs \overline{\eta}.$ be the geometric generic fiber.
Then, it admits a minimal resolution $X' \rar X \subs \overline{\eta}.$.
Notice that $X'$ is not isomorphic to $\pr 2.$.
This morphism is defined over a finite extension of $K(Z)$.
Therefore, up to a generically finite base change of $Z$, we may assume that the minimal resolution of $X \subs \overline{\eta}.$ is defined over $K(Z)$.
Thus, we can replace $X$ with the corresponding blow-up resolving the generic fiber.
Notice that in this process, $B^h$ remains effective.
Thus, from now on, we may assume that $X_\eta$ is smooth.

By the assumptions of the theorem, we have that the Kodaira dimension of the generic fiber satisfies $\kappa (X _\eta) \leq 0$.
In case $\kappa (X _\eta) = 0$, we have $B^h=0$, and the statement follows from work of Fujino \cite[Lemma 4.1, Corollary 6.4]{Fuj03}.

Therefore, we may assume $\kappa (X_\eta)<0$, and that $X_\eta$ is not isomorphic to $\pr 2.$.
Hence, either the minimal model of the geometric generic fiber $X_{\overline{\eta}}$ is a minimal ruled surface over a curve \cite[Theorem 1.29]{KM}, or $X_{\overline{\eta}}$ maps to the blow-up of $\pr 2.$ at one point.
In both cases, the geometric generic fiber admits a morphism to a curve.
Thus, up to a base change of the fibration by a generically finite morphism, we may assume that the generic fiber itself admits a morphism to a curve.
In particular, we have a commutative diagram of rational maps

\begin{center}
\begin{tikzpicture}
\matrix(m)[matrix of math nodes,
row sep=2.6em, column sep=2.8em,
text height=1.5ex, text depth=0.25ex]
{X & Y\\
& Z\\};
\path[->,font=\scriptsize,>=angle 90]
(m-1-1) edge[dashed] node[auto] {$g$} (m-1-2)
edge node[auto] {$f$} (m-2-2)
(m-1-2) edge[dashed] node[auto] {$h$} (m-2-2);
\end{tikzpicture}
\end{center}
where $g$ and $h$ are actual morphisms over an open subset of $Z$.
Thus, up to taking birational modifications of $X$ and $Y$ that are isomorphism over the generic point of $Z$, we may assume that $g$ and $h$ are morphisms.
Then, up to taking the normalization of $Y$ in $Z$, we may assume that $g$ and $h$ have connected fibers.
Notice that, by construction, the relative dimensions of $g$ and $h$ are both 1, and the generic fibers are smooth.
Furthermore, the generic fiber of $g$ is rational.

In order to better understand the above picture, we consider the general fibers.
As all three morphisms have smooth generic fibers, the general fibers are smooth as well.
We denote them by $F$, $G$ and $H$ respectively.
The curve $G$ is isomorphic to $\pr 1.$, and is contained in the surface $F$.
Since $F$ is general, the sub-klt sub-pair $(X,B)$ induces a sub-klt sub-pair $(F,B_F)$ \cite[Corollary 9.5.6]{LAZ2}.
Since $B$ is effective along $X_\eta$, $B_F$ is effective too, and $(F,B_F)$ is klt.
In particular, we have $\K F. + B_F \sim_\Q 0$.
As we have the morphism $F \rar H$, we can apply the canonical bundle formula in this setup. It will produce an effective boundary divisor $B_H$ and a nef divisor $M_H$ such that $\K H. + B_H + M_H \sim_\Q 0$.
This implies that the genus of $H$, and hence of the generic fiber of $h$, is either 0 or 1.

Notice that, to obtain the morphism $X \rar Y$, we did not blow up any horizontal stratum, as the morphism of generic fibers was well defined up to a base change.
Therefore, in the new model $X$, the divisor $B$ is effective over a dense open set of $Z$.
Then, by Corollary \ref{generalized ps conj curves}, the moduli b-divisor $\bM Y.$ is b-semi-ample.
Also, by Proposition \ref{inv adj fiber spaces}, the boundary part $\bB Y,Y.$ is effective, and $(Y,\bB Y,Y. + \bM Y,Y.)$ is generalized klt over the generic point of $Z$.
Therefore, we can apply Corollary \ref{generalized ps conj curves} to the fibration $Y \rar Z$.
Furthermore, by Lemma \ref{lemma diagram of fibrations}, the moduli, and boundary b-divisors induced by $Y \rar Z$ and by $X \rar Z$ agree.
Therefore, we conclude that $\bM Z.$ is b-semi-ample.
\end{proof}

\begin{remark}{\em
The proof of Theorem \ref{prokhorov shokurov rel dim 2} does not imply effective b-semi-ampleness, because in the course of the proof we replaced the base of the fibration with a generically finite cover.
One would need a bound on the degree of the cover in order to achieve effectivity.}
\end{remark}

\begin{proof}[Proof of Theorem \ref{prokhorov shokurov rel dim 2 revised}]
{
Let $B^h$ denote the horizontal part of $B$.
Let $(X^m,\Delta^m)$ be a $\Q$-factorial dlt model of $(X,B^h)$, and let $B^m$ the log-pullback of $B$ to $X^m$.
Since $(X,B)$ is klt over the generic point of $Z$, we have that $(X^m,(B^m)^h)$ is klt.
Then, let $X' \rar X^m$ be a terminalization of $(X^m,(B^m)^h)$.
Let $B'$ denote the log-pullback of $B$ to $X'$.
Notice that, by construction, we have $(B')^h \geq 0$.}

{
Since $(X,B)$ is not terminal, the morphism $X'_\eta \rar X_\eta$ is not an isomophism.
In particular, $X_{\overline{\eta}}$ is not isomorphic to $\pr 2.$.
Therefore, Theorem \ref{prokhorov shokurov rel dim 2} applies to $(X',B') \rar Z$, and the claim follows.}
\end{proof}

The approach in Theorem \ref{prokhorov shokurov rel dim 2} suggests that the proof of part (i) of Conjecture \ref{Prokhorov-Shokurov conjecture} can be reduced to two extreme cases: Mori fiber spaces, and cases when $X_\eta$ is Calabi--Yau.
If we use techniques from the MMP, we avoid generically finite base changes, and we can also address part (iii) of Conjecture \ref{Prokhorov-Shokurov conjecture}.

\begin{theorem}
Fix a natural number $n$.
Assume that part (i) (or (iii)) of Conjecture \ref{Prokhorov-Shokurov conjecture} is true if the relative dimension is strictly less than $n$.
Then, part (i) (respectively (iii)) of Conjecture \ref{Prokhorov-Shokurov conjecture} in relative dimension $n$ can be reduced to the two following cases:
\begin{itemize}
\item $f  \colon  X \rar Z$ is a $\K X.$-Mori fiber space; or
\item $\K X. \sim_{\Q,Z} 0$, and $B^h=0$.
\end{itemize}
\end{theorem}

\begin{proof}
{As argued in the proof of Theorem \ref{prokhorov shokurov rel dim 2}, we may assume that the varieties involved are projective.}
By \cite[Theorem 3.1]{KK10}, we can assume that $X$ is $\Q$-factorial, and $(X,B^h)$ is klt.
Thus, we can run a $\K X.$-MMP over $Z$ with the scaling of an ample divisor.
Notice that there are two cases: either $B^h=0$ or $B^h>0$.
In the former case, the MMP terminates with a good minimal model for $\K X.$ by \cite[Theorem 1.4]{Bir12} and \cite[Theorem 1.1]{HMX14b}.
In the latter case, the MMP terminates with a Mori fiber space by \cite{BCHM}.
Call $g \colon  X' \rar Z$ the final model reached by the MMP, and let $(X',B')$ be the sub-pair induced by $(X,B)$.

First, assume that $X'$ is a good minimal model for $X$.
Then, $\K X'.\sim \subs \Q,Z. 0$.
Since $\K X'. + B' \sim \subs \Q,g. 0$ and $B^h=0$, it follows that $B'$ is the pullback of a divisor on $Z$.

Now, assume that the MMP terminates with a Mori fiber space $h \colon  X' \rar Y$.
If $Y$ is birational to $Z$, we are done. Hence, we may assume that $\dim X' -\dim Y < n$.
Then, the fibration $(X',B') \rar Y$ satisfies the conditions of Conjecture \ref{Prokhorov-Shokurov conjecture} for a smaller relative dimension.
By assumption, a generalized pair structure $(Y,\bB Y. + \bM Y.)$ is induced on $Y$, where $\bM Y.$ is b-semi-ample (respectively effectively b-semi-ample).
Thus, we can apply the lower dimensional case of Conjecture \ref{generalized conjecture}, which holds by Theorem \ref{PS conjecture implies generalized conj}, to the fibration $Y \rar Z$, and induce a generalized pair structure $(Z,\bB Z. + \bM Z.)$ on $Z$, where $\bM Z.$ is b-semi-ample (respectively effectively b-semi-ample).
By Lemma \ref{lemma diagram of fibrations}, the generalized pair $(Z,\bB Z. + \bM Z.)$ induced this way is the same as the one induced by $(X,B)$.
Thus, the claim follows.
\end{proof}

\printbibliography

\Addresses

\end{document}